\documentclass[letterpaper,12pt]{article}
\usepackage{natbib} \bibpunct{(}{)}{;}{a}{}{,}
\usepackage{fullpage}
\usepackage{amsmath,amssymb,amsthm}
\usepackage{enumerate}
\usepackage{appendix}
\usepackage{graphicx,subfig}
\usepackage{color}
\usepackage{hyperref}
\hypersetup{
        colorlinks   = true,
        linkcolor    = blue,
        citecolor    = green
}
\usepackage{authblk}

\newtheorem{theorem}{Theorem}
\newtheorem{corollary}{Corollary}
\newtheorem{lemma}{Lemma}
\newtheorem{observation}{Observation}
\newtheorem{definition}{Definition}
\newtheorem{remark}{Remark}
\newtheorem{example}{Example}
\newtheorem{assumptionu}{Assumption}
\newtheorem{assumptionw}{Assumption}

\newcommand{\E}{\mathbb{E}}
\newcommand{\bbO}{\mathbb{O}}
\newcommand{\bbR}{\mathbb{R}}
\newcommand{\bbW}{\mathbb{W}}

\newcommand{\bM}{M}

\newcommand{\dbar}{\bar{d}}

\newcommand{\dcirc}{\mathring{d}}

\newcommand{\calC}{\mathcal{C}}

\newcommand{\calL}{\mathcal{L}}
\newcommand{\calN}{\mathcal{N}}

\newcommand{\calS}{\mathcal{S}}
\newcommand{\calT}{\mathcal{T}}
\newcommand{\calX}{\mathcal{X}}

\newcommand{\bMtilde}{\tilde{\bM}}

\newcommand{\Ptilde}{\tilde{P}}
\newcommand{\Utilde}{\tilde{U}}

\newcommand{\Ytilde}{\tilde{Y}}

\newcommand{\nutilde}{\tilde{\nu}}

\newcommand{\ahat}{\hat{a}}
\newcommand{\bhat}{\hat{b}}

\newcommand{\Ahat}{\hat{A}}

\newcommand{\Fhat}{\hat{F}}
\newcommand{\Mhat}{\hat{M}}
\newcommand{\Phat}{\hat{P}}
\newcommand{\Shat}{\hat{S}}
\newcommand{\Uhat}{\hat{U}}
\newcommand{\Xhat}{\hat{X}}

\newcommand{\rhohat}{\hat{\rho}}

\newcommand{\Astar}{A^*}

\newcommand{\UA}{\Uhat}
\newcommand{\UP}{U}
\newcommand{\SA}{\Shat}
\newcommand{\SP}{S}

\newcommand{\ivec}{c}

\newcommand{\tr}{\operatorname{tr}}
\newcommand{\supp}{\operatorname{supp}}
\newcommand{\diag}{\operatorname{diag}}
\newcommand{\iid}{\stackrel{\text{i.i.d.}}{\sim}}
\newcommand{\tti}{2,\infty}
\newcommand{\inlaw}{\xrightarrow{\calL}}
\newcommand{\inprob}{\xrightarrow{P}}
\newcommand{\indic}{\mathbb{I}}

\newcommand{\Bern}{\operatorname{Bern}}
\newcommand{\Var}{\operatorname{Var}}

\newcommand{\opBeta}{\operatorname{Beta}}
\newcommand{\Multinomial}{\operatorname{Multinomial}}
\newcommand{\RDPG}{\operatorname{RDPG}}
\newcommand{\ASE}{\operatorname{ASE}}

\newcommand{\GM}{\operatorname{GM}}

\newcommand{\UBF}{V}
\newcommand{\UBFhat}{\hat{V}}

\newcommand{\dGM}{d_{\GM}}
\newcommand{\WGM}[1]{W_{#1}}

\begin{document}

\title{Bootstrapping Networks with Latent Space Structure}
\author[1]{Keith D. Levin}
\author[2]{Elizaveta Levina}
\date{}

\affil[1]{\small Department of Statistics, University of Wisconsin-Madison, USA.}
\affil[2]{\small Department of Statistics, University of Michigan, USA.}

\maketitle

\begin{abstract}
A core problem in statistical network analysis is to develop network analogues of classical techniques.  The problem of bootstrapping network data stands out as especially challenging, since typically one observes only a single network, rather than a sample. Here we propose two methods for obtaining bootstrap samples for networks drawn from latent space models. The first method generates bootstrap replicates of network statistics that can be represented as U-statistics in the latent positions, and avoids actually constructing new bootstrapped networks.  The second method generates bootstrap replicates of whole networks, and thus can be used for bootstrapping any network function.  Commonly studied network quantities that can be represented as U-statistics include many popular summaries, such as average degree and subgraph counts, but other equally popular summaries, such as the clustering coefficient, are not expressible as U-statistics and thus require the second bootstrap method.  Under the assumption of a random dot product graph, a type of latent space network model, we show consistency of the  proposed bootstrap methods. We give motivating examples throughout and demonstrate the effectiveness of our methods on synthetic data.  \end{abstract}

\section{Introduction} \label{sec:intro}
As network data become ever more common in the sciences, a need has emerged
for network analogues of classical statistical methods.
Such analogues have been developed for some tasks such as
network two-sample testing \citep{FosHof2015,TanAthSusLyzPri2017,TanAthSusLyzParPri2017},
changepoint detection \citep{WanYuRin2018}
and community estimation
\citep[i.e., the network analogue of clustering;][]{LyzSusTanAthPri2014,Abbe2018},
to name a few, but many other common tasks still have no network equivalents, or have not been rigorously studied.
One such task is the generation of bootstrap samples.
The bootstrap \citep{EfrTib1994}
allows one to make inferences about a population distribution by resampling
from an observed i.i.d.\ sample.
Unfortunately, the fact that we typically observe only one network has made developing network analogues difficult, though there are resampling methods for other dependent data such as time series  \citep[see, e.g.,][]{Lahiri2003},
and methods for the related task of cross-validation on networks have been developed \citep[see][and citations therein]{LiLevZhu2016}.

We propose two related bootstrap methods for networks, 
one for generating bootstrap samples of certain network statistics 
(e.g., subgraph counts, centrality measures, etc.),
and another for generating bootstrap replicates of whole networks.  Obviously the latter method can also be used to bootstrap network statistics, but it comes with a higher computational cost, and thus it is convenient to have the first method,  which applies to many commonly used network statistics.
The topic of bootstrapping network data is still new in the literature,
and we are aware of only a few papers on the topic \citep{BhaBic2015,GreSha2017,LunSar2019}, largely focused on the problem of bootstrapping subgraph counts of networks generated from graphons.
For a particular subgraph of interest on $p$ vertices, \cite{BhaBic2015} generate bootstrap replicates of its counts by sampling from the set of connected $p$-node subgraphs of the observed network on $n > p$ nodes, using the algorithm introduced in \cite{Wernicke2006}.
Following similar lines, \cite{GreSha2017} consider two related
methods for generating bootstrap samples of subgraph counts from
an observed network. The first relies upon resampling from what the
authors term the {\em empirical graphon}, which amounts to
resampling vertices with replacement from the observed network.
The second method relies upon resampling from a stochastic block model
fit to the observed network, since such a model can 
approximate any graphon, if the number of communities
is chosen sufficiently large.
\cite{ChaKolYao2018} considered the problem of estimating subgraph densities under a different setting, where there exists a true underlying network, but one observes a noisy version based on adding or removing edges independently.
The authors presented a bootstrap sampling scheme for constructing confidence intervals for the subgraph densities of the true underlying network.

Subgraph counts are no doubt important, but there are many other network quantities of interest, and little is known about bootstrapping anything other than subgraph counts.
Under network latent space models \citep{HofRafHan2002}, many such quantities of interest can be expressed as U-statistics of the latent positions.
We show (Theorem~\ref{thm:Uboot}) that under the fairly general latent space model known as the {\em random dot product graph} \citep{YouSch2007}, these U-statistics can be bootstrapped by first estimating the latent positions of the vertices and then bootstrapping a plug-in version of the quantity of interest using known techniques for bootstrapping U- and V-statistics \citep{ArcGin1992,HusJan1993,BosCha2018}.
Recent independent work by \cite{LunSar2019}, concurrent with this manuscript, presents a bootstrapping procedure based on sampling induced subgraphs of the observed network, broadly similar to the approach in \cite{BhaBic2015}.
They showed that their bootstrap is valid for any network statistic obeying a central limit theorem under appropriate scaling and subgraph sampling operations, under a condition that is nontrivial to verify in general; our results in this paper can be adapted to show that this condition is satisfied by the U-statistics we consider.
However, since the bootstrap introduced in \cite{LunSar2019} operates by sampling subgraphs of the observed network, it cannot take advantage of the computational speedups available under our U-statistic bootstrapping scheme, discussed in Section~\ref{sec:ustat}.

In another direction, there are settings in which we may wish to generate bootstrap samples of whole networks, either for use in downstream inference or to bootstrap other network statistics that do not admit a U-statistic representation.
That is, we seek to use a single observed network to generate bootstrap network samples that, ideally, have the same distribution as the observed network, under as general assumptions as possible.
We show in Theorem~\ref{thm:RDPGboot:wasser} that this is possible under the random dot product graph, which allows for arbitrary distributions of latent positions and includes many other commonly used models with independent edges as special cases, including the stochastic block model.
Previous work has considered generation of parametric bootstrap network samples from the stochastic block model \citep{BicChoChaZha2013,BicSar2015,Lei2016}.
\cite{ShaAst2017} proved that in latent space models, generating network samples based on maximum likelihood estimates of the latent positions yields consistent bootstrap samples, but obtaining these maximum likelihood estimates is, under most latent space models, computationally infeasible.
\cite{Lei2018graphs} studied a latent space model for exchangeable random graphs.
That paper does not consider bootstrapping, but it contains a result analogous to our Theorem~\ref{thm:RDPGboot:wasser} for a weaker notion of graph convergence.
To the best of our knowledge, the present paper is the first to propose a computationally feasible network bootstrap under general latent space models.   

\subsection{Latent Space Models and U-statistics}

As a simple example of how latent space models interact nicely with certain network statistics, consider the triangle density, defined as
\begin{equation*}
\Phat(K_3) = \binom{n}{3}^{-1} \sum_{1 \le i < j < k \le n}
		A_{ij} A_{jk} A_{ki},
\end{equation*}
where $A \in \{0,1\}^{n \times n}$ is a symmetric adjacency matrix.
Subgraph densities such as this play a role for node-exchangeable random graphs (i.e., those generated from graphons) that is analogous to the role of moments for Euclidean data \citep{BicCheLev2011,Lovasz2012,MauOlhPriWol2017}.
As a result, obtaining confidence intervals for expected subgraph densities such as $\E \Phat(K_3)$ based on a single observed graph is of import for many questions of statistical interest.

Under a latent space model, $A$ is generated by first drawing latent positions $X_1,X_2,\dots,X_n$ i.i.d.\ from some distribution $F$ on a set $\calX$ endowed with a symmetric link function $\kappa : \calX \times \calX \rightarrow [0,1]$.
Conditioned on the latent positions, the entries of $A$ are drawn independently, with $A_{ij} \sim \Bern( \kappa(X_i,X_j) )$.
The expectation of the triangle density conditional on the latent positions is
\begin{equation*}
\E[ \Phat(K_3) \mid X_1,\dots,X_n ] =
	\binom{n}{3}^{-1} \sum_{1 \le i < j < k \le n}
		\kappa(X_i,X_j) \kappa(X_j, X_k) \kappa(X_k, X_i),
\end{equation*}
which is a U-statistic with the kernel $h(x,y,z) = \kappa(x,y) \kappa(y,z) \kappa(z,x)$.
In Section~\ref{sec:ustat}, we show that a number of popular network quantities can be similarly written as U-statistics, including, notably, all subgraph densities.
For any such U-statistic with kernel $h$,  if we  had access to the latent positions, we could apply existing techniques for bootstrapping U-statistics, such as, for example, \cite{BicFre1981,ArcGin1992,HusJan1993,BosCha2018}.
This would permit us to calculate bootstrap confidence intervals for the network parameter $\theta = \E h(X_1,X_2,\dots,X_m)$ by drawing bootstrap replicates of
\begin{equation*}
U_n = U_n(h) = \binom{n}{m}^{-1} \sum_{1 \le i_1 < i_2 < \cdots < i_m \le n}
				h(X_{i_1},X_{i_2},\dots,X_{i_m}).
\end{equation*}
Of course, in practice we do not observe the latent positions and must instead estimate them from the observed adjacency matrix $A$.
Supposing for now that we had estimates $\Xhat_1,\Xhat_2,\dots,\Xhat_n$ based on $A$, a sensible approach would be to instead bootstrap the quantity
\begin{equation*}
\Uhat_n = \Uhat_n(h) =
		 \binom{n}{m}^{-1} \sum_{1 \le i_1 < i_2 < \cdots < i_m \le n}
                        h(\Xhat_{i_1},\Xhat_{i_2},\dots,\Xhat_{i_m}).
\end{equation*}
Under suitable smoothness conditions on $h$ (e.g., that it admits a Taylor
expansion on the support of $F$) and provided that the true latent positions can be sufficiently accurately estimated, we may reasonably expect $\Uhat_n$ to be a good approximation to $U_n$ and that, furthermore, bootstrap techniques applied to $\Uhat_n$ instead of $U_n$ will still produce an (approximately) equivalent bootstrap sampling distribution.

\subsection{Nonparametric network bootstrap samples}
\label{subsec:intro:netboot}

The classical bootstrap \citep{EfrTib1994} is based upon sampling with replacement from an observed sample $X_1,X_2,\dots,X_n$ as though it were the population itself.
Under suitable conditions, these bootstrap samples can be used to make inferences about the population distribution, but there is no straightforward analogue for network data, owing both to the fact that we typically observe only one network rather than an i.i.d.\ sample and to the fact that there is dependence among the edges.
Fortunately, the structure of latent space models provides a way forward.

The latent positions $X_1,X_2,\dots,X_n$ are drawn i.i.d.\ from the distribution $F$, and thus the bootstrap sample 
$X_1^*,X_2^*,\dots,X_n^*$ drawn i.i.d.\ from the empirical distribution $F_n = n^{-1} \sum_{i=1}^n \delta_{X_i}$
can be thought of as being approximately drawn from $F$ so long as the empirical distribution $F_n$ is a good approximation to $F$.
Thus, once again, if we had access to the latent positions, it would be natural to generate a bootstrap replica $\Astar$ of the adjacency matrix by drawing from $F_n$ as though it were the latent position distribution.
That is, conditional on the sample from $F$, we draw $X_1^*,X_2^*,\dots,X_n^*$ i.i.d.\ from $F_n$ and generate $\Astar_{ij} \sim \Bern( \kappa(X_i^*,X_j^*) )$.
Since we do not observe the latent positions $X_1,X_2, \dots, X_n$,  a natural approach is to produce latent position estimates $\Xhat_1,\Xhat_2,\dots,\Xhat_n$ from $A$, and use their empirical distribution $\Fhat_n$ in place of $F_n$.
Then we can generate bootstrap network samples $\Ahat^*$ by drawing from a latent space model with latent position distribution $\Fhat_n$.
Once again, provided that the estimates are good approximations to the true latent positions, we may expect $\Fhat_n$ to be a good approximation to $F_n$, which is in turn a good approximation to $F$.
We may then expect that the distribution of $\Ahat^*$ is a good
approximation to the distribution of $A$.
Indeed, we will prove below that if $H$ is an independent copy of $A$,
$\Ahat^*$ and $H$ converge in a suitably-defined Wasserstein metric.

Defining a Wasserstein distance on graphs first requires a notion of
distance on graphs.
There are a few such distances in the literature,
most notably the cut metric \citep{Lovasz2012}.
In Section~\ref{sec:netboot}, 
we consider a different notion of distance on graphs,
which we call the {\em graph matching distance}.
The graph matching distance is an upper bound on the cut metric
and is designed to take more global information into account,
rather than capturing only the local information conveyed by
a single graph cut.
Working under a particular latent space model, we will show that our 
network resampling scheme outlined above does indeed produce
bootstrap samples $\Ahat^*$
that converge to the distribution of $A$ in the Wasserstein distance
defined with respect to the graph matching distance.

Finally, we note that while we restrict our attention in this paper to the random dot product graph for the sake of concreteness and notational simplicity, the basic ideas of this paper are applicable for general latent space models, as long as the latent positions can be estimated at a suitable rate.
This point is discussed in more detail at the end of Section \ref{subsec:prelim:RDPG}. 

The rest of this paper is structured as follows.
In Section~\ref{sec:prelim}, we briefly review the necessary background related to the random dot product graph and U-statistics.
Section~\ref{sec:ustat} presents our method and theoretical results for bootstrapping network U-statistics, and Section~\ref{sec:netboot} covers our method for generating bootstrap samples of whole networks.
Section~\ref{sec:expts} gives a brief experimental demonstration of both of these methods.  We conclude in Section~\ref{sec:conclusion} with a brief discussion and directions for future work.

\section{Background and Preliminaries} \label{sec:prelim}

In this section, we provide a brief review of the random dot product
graph and its basic properties, as well as necessary background
related to U-statistics and the bootstrap. We start by establishing notation.  

\subsection{Notation}
For a positive integer $n$, let $[n]$  denote the set $\{1,2,\dots,n\}$ and let $S_n$ denote the set of permutations of $[n]$.   
For two integers $m < n$, let $\calC^n_m$ denote the set of all ordered $m$-tuples of distinct elements of $[n]$,
\begin{equation*}
\calC^n_m 
= \{ (i_1,i_2,\dots,i_m) \in [n]^m : 1 \le i_1 < i_2 < \dots < i_m \le n \}.
\end{equation*}
We use the superscript $T$ to denote vector or matrix transpose.  For a vector $x$, we use
$\| x \|$ to denote its Euclidean norm.
For a matrix $M$, we use  $\| M \|$ to denote the spectral norm,
$\|M\|_F$ to denote the Frobenius norm, and
$\|M\|_1 = \sum |M_{ij}|$.  Throughout, we use $C$ to denote a generic positive constant, not depending on the network size, whose value may change from line to line or within the same line.

\subsection{The Random Dot Product Graph} \label{subsec:prelim:RDPG}

The random dot product graph \citep[RDPG][]{YouSch2007,AthFisLevLyzParQinSusTanVogPri2018} is a a latent space network model \citep{HofRafHan2002} in which the latent positions are points in Euclidean space, and edge probabilities are given by inner products of the latent positions.

\begin{definition}[Random dot product graph]
Let $F$ on $\bbR^d$ be a $d$-dimensional {\em inner product distribution}, meaning that $0 \le x^Ty \le 1$ for all $x,y \in \supp F$.
Let $X_1,X_2,\dots,X_n$ be drawn i.i.d.\ from $F$ and arranged as the rows of $X \in \bbR^{n \times d}$.
Conditional on $X$, generate the symmetric adjacency matrix $A \in \bbR^{n \times n}$ by independently drawing $A_{i,j} \sim \Bern(X_i^T X_j)$ for all $1 \le i < j \le n$.
We say that $A$ is a RDPG with latent position distribution $F$, and write $(A,X) \sim \RDPG(F,n)$.
\end{definition}

We note immediately that the RDPG is not fully identifiable, since any orthogonal transformation of the latent positions $X$ preserves inner products and yields the same distribution for $A$.   
Thus, we can only hope to recover the latent positions up to an 
orthogonal transformation, and we will therefore only consider U-statistics that are invariant to such orthogonal transformations.
Non-identifiability of this type is essentially unavoidable in latent
space models; see \cite{ShaAst2017}.  
Throughout, we will assume without loss of generality
that the second moment matrix of $F$,
$\Delta = \E X_1X_1^T \in \bbR^{d \times d}$ is of full rank,
since otherwise we may restrict ourselves to an equivalent
lower-dimensional model in which $\Delta$ has full rank.

The main appeal of the RDPG relative to other latent space models is that the latent positions can be estimated by spectral methods.    The truncated singular value decomposition of the adjacency matrix gives accurate (up to an orthogonal transformation) 
estimated latent positions, referred to as the {\em adjacency spectral embedding} in the RDPG literature \citep{SusTanFisPri2012}.

\begin{definition}[Adjacency spectral embedding]
Let $\SA \in \bbR^{d \times d}$ be the diagonal matrix formed by the
top $d$ largest-magnitude eigenvalues of the adjacency matrix $A$,
and let $\UA \in \bbR^{n \times d}$ be the matrix with the corresponding eigenvectors as its columns.  The adjacency spectral embedding of $A$ is defined as $\ASE(A,d) = \UA \SA^{1/2} \in \bbR^{n \times d}$.
\end{definition}

The rows of $\Xhat = \ASE(A,d)$ are estimates of the latent positions, with a guaranteed convergence rate. 
\begin{lemma}[\cite{LyzSusTanAthPri2014}, Lemma 5]
\label{lem:2toinfty}
Let $(A,X) \sim \RDPG(F,n)$.
Then with
probability at least $1-Cn^{-2}$ there exists an orthogonal
matrix $Q \in \bbR^{d \times d}$ such that
\begin{equation*}
\max_{i \in [n]} \| Q^T \Xhat_i - X_i \| \le \frac{ C \log n }{ \sqrt{n} }. 
\end{equation*}
\end{lemma}

Lemma~\ref{lem:2toinfty} provides evidence that the plug-in procedure sketched out 
in Section~\ref{sec:intro} may succeed in the limit.
We will formalize this intuition in Section~\ref{sec:ustat}

A primary drawback of the RDPG as defined above is that it produces
only networks whose expected adjacency matrices are positive definite.
This restriction can be removed by instead considering the
generalized RDPG \citep{RubPriTanCap2017}, in which the inner product
$X_i^T X_j$ is replaced with $X_i^T I_{p,q} X_j$,
where $I_{p,q}$ is a diagonal matrix with $p$ ones and $q$ negative ones
on its diagonal.
The rooted graph distribution \citep{Lei2018graphs} further generalizes
this model to allow for latent positions residing in
an infinite-dimensional Kre\v{\i}n space.
In both models, an eigenvalue truncation quite similar
to the ASE defined above recovers the latent positions uniformly
(subject to certain spectral decay assumptions in the case of
the rooted graph distribution).
While both of these generalizations are useful, they come at the
expense of added notational complexity that would not add to the
core ideas of this paper.
Thus, we restrict our attention here to the RDPG, while noting that our
results can be extended with minimal additional assumptions
to these two more general models.

\subsection{Bootstrapping U-statistics}
\label{subsec:prelim:ustat}

Given a measurable function $h : \calX^m \rightarrow \bbR$
symmetric in its $m$ arguments
and a sample $X_1,X_2,\dots,X_n$ drawn i.i.d.\ from some distribution
$F$ on $\calX$, the U-statistic with kernel $h$ is given by
\begin{equation} \label{eq:def:ustat}
U_n = U_n(h) = \frac{1}{\binom{n}{m}}
		\sum_{c \in \calC^n_m} h(X_{i_1},X_{i_2,},\dots,X_{i_m}).
\end{equation}
The study of quantities of this form dates to
\cite{Hoeffding1948}, and we refer the interested reader
to Chapter 5 of \cite{Serfling1980} for a more thorough overview.
Suppose that $\theta(F) = \E h(X_1,X_2,\dots,X_m)$
is some parameter or quantity of interest,
where the expectation is taken with respect to $X_1,X_2,\dots,X_m \iid F$,
A classic result states that if $X_1,X_2,\dots,X_n \iid F$, then
provided the kernel $h$ is non-degenerate with respect to $F$
(i.e., that $h_1(z) = \E h(z,X_2,X_3,\dots,X_m)$ is not a constant in $z$),
$\sqrt{n}(U_n - \theta(F))$ converges in distribution to a
mean-$0$ normal with variance $m^2 \zeta_1$, where
\begin{equation} \label{eq:def:zeta}
  \zeta_1  
        = \E\left(\E[ h(X_1,X_2,\dots,X_m) \mid X_1 ] - \theta \right)^2.
\end{equation}
Since this quantity is typically unknown, in order to
obtain a confidence interval for $\theta(F)$,
the bootstrap can be used to approximate the sampling distribution of $U_n$.
Bootstrapping U-statistics has received much attention
in the literature
\citep[see, e.g.,][to name just a few]{BicFre1981,ArcGin1992,HusJan1993,BosCha2018}.
A bootstrap for non-degenerate U-statistics was introduced in
\cite{BicFre1981}, who showed that 
as long as $\E h(X_1,X_1,\dots,X_1) < \infty$, one can draw
$X_1^*,X_2^*,\dots,X_n^* \iid F_n$ and consider the estimate
\begin{equation*}
\binom{n}{m}^{-1} \sum_{\ivec \in \calC^n_m}
        h(X_{i_1}^*,X_{i_2}^*,\dots,X_{i_m}^*).
\end{equation*}
As pointed out in \cite{BicFre1981},
simply resampling from $F_n$ in this way fails when the kernel $h$ is
degenerate, but a weighted bootstrap can be used instead
\citep{ArcGin1992,HusJan1993}.
While we do not consider degenerate kernels here,
weighted bootstrap schemes can also yield substantial computational
speedups. Thus, following \cite{BosCha2018}, we consider the quantity
\begin{equation} \label{eq:def:efron}
U^*_n =
\binom{n}{m}^{-1} \sum_{c \in \calC^n_m}
        \bbW_{c}~
        h(X_{i_1},X_{i_2},\dots,X_{i_m}),
\end{equation}
where $\bbW \in \bbR^{\calC^n_m}$ is a vector of random weights.
Taking $\bbW_{c} = \prod_{k=1}^m W_{i_k}$
with $W \sim \Multinomial(n,n^{-1})$, one recovers the ``Efron-weighted''
bootstrap \citep{ArcGin1992,HusJan1993}.
Taking
\begin{equation*}
  \bbW_{c} = \frac{W_{i_1} + \cdots + W_{i_m}}{m\sqrt{1-1/n}},
\end{equation*}
one obtains the
additive bootstrap discussed at length in Chapter 4 of \cite{BosCha2018}.
Under suitable conditions on the weight vector $\bbW$,
the quantity in Equation~\eqref{eq:def:efron} converges in distribution
to the same $N(0,m^2 \zeta_1)$ limit as $U_n$.
The specific conditions on $\bbW$ needed to ensure this convergence
vary from one paper to another, but provided those conditions are
met, $U^*_n$ is {\em distributionally consistent},
in the sense that  the distribution of the bootstrap
sample $\sqrt{n}(U^*_n-\theta)$ matches that of $\sqrt{n}(U_n-\theta)$
in the limit as $n \rightarrow \infty$.








%

\section{Bootstrapping Latent Position U-statistics}
\label{sec:ustat}

Our aim in this section is to obtain bootstrap samples of
a U-statistic $U_n = U_n(h)$,
which is a function of the latent positions $X_1,X_2,\dots,X_n$.
The obstacle is that we only observe the adjacency matrix $A$, not the latent positions themselves. We will show below that using the ASE estimates $\Xhat_1,\Xhat_2,\dots,\Xhat_n$ in place of the true latent positions results in a U-statistic that converges to $U_n$ almost surely. Further, bootstrapping the resulting plug-in U-statistic yields asymptotically equivalent bootstrap samples to what we would have obtained
by following the schemes in Section~\ref{subsec:prelim:ustat} if the latent positions were observable.   Before presenting these convergence results,
we highlight a few examples of network quantities that are expressible
as U-statistics in the latent positions under the RDPG model.

\subsection{Network U-statistics: Examples}
\begin{example}[Average Degree]
\normalfont
Consider the (normalized) average degree,
\begin{equation*}
\dbar(A) = \frac{1}{n} \sum_{i=1}^n \frac{ d_i }{n-1}
        = \frac{1}{n} \sum_{i=1}^n \frac{ \sum_{j\neq i} A_{i,j} }{n-1} \ . 
\end{equation*}
Under the RDPG, its conditional expectation is 
\begin{equation*} \begin{aligned}
\E[ \dbar(A) \mid X] 
&= \E \frac{1}{n(n-1)} \sum_{i=1}^n \sum_{j\neq i} \E[ A_{i,j} \mid X_i,X_j ]
= 2\binom{n}{2}^{-1} \sum_{i < j} X_i^T X_j,
\end{aligned} \end{equation*}
which is a U-statistic
with kernel $h(x,y) = 2x^T y$.
\end{example}

\begin{example}[Subgraph Counts]
\normalfont
Let $R$ and $G$ be graphs on $m$ and $n$ vertices, respectively,
with $m \le n$.
Numbering the vertices of $G$ arbitrarily,
for $c = (i_1,i_2,\dots,i_m) \in \calC^n_m$,
let $G[c]$ denote the $m$-vertex subgraph of $G$ induced by
$i_1,i_2,\dots,i_m$ and consider the quantity
\begin{equation} \label{eq:def:sgcount}
  \Phat(R) = \frac{1}{\binom{n}{m}}
		\sum_{c \in \calC^n_m} \indic\{ G[c] \simeq R \},
\end{equation}
where we write $H \simeq R$ to denote that
graphs $H$ and $R$ are isomorphic.
$\Phat(R)$ thus measures the (empirical) proportion of times that $R$ appears as a subgraph
out of the total number of possible subgraphs on $m$ vertices.
Letting $B \in \bbR^{m \times m}$ be the adjacency matrix of graph $R$,
we can write
\begin{equation*}
\E[ \Phat(R) \mid X ] = \frac{1}{\binom{n}{m}}
		\sum_{c \in \calC^n_m}
		\sum_{\tau \in S_m}
		\frac{ \prod_{1 \le k < \ell \le m}
		(X_{i_{\tau(k)}}^T X_{i_{\tau(\ell)}})^{B_{k\ell}}
		(1-X_{i_{\tau(k)}}^T X_{i_{\tau(\ell)}})^{1-B_{k\ell}} }
		{ N(R) }
\end{equation*}
where $N(R)$ denotes the number of graphs isomorphic to $R$.
From this, it is easy to see that
$\E[ \Phat(R) \mid X ]$ is a U-statistic with kernel
\begin{equation*}
h_R(x_1,x_2,\dots,x_m) =
	\frac{1}{N(R)} \sum_{\tau \in S_m}
                \prod_{1 \le k < \ell \le m}
                (x_{i_{\tau(k)}}^T x_{i_{\tau(\ell)}})^{B_{k\ell}}
                (1-x_{i_{\tau(k)}}^T x_{i_{\tau(\ell)}})^{1-B_{k\ell}}.
\end{equation*}
\end{example}

\begin{example}[Maximum Mean Discrepancy] \normalfont\normalfont
\normalfont
The maximum mean discrepancy \citep[MMD;][]{GreBorRasSchSmo2012}
is a test statistic for nonparametric two-sample hypothesis testing.
Given $\lambda \in [0,1]$ and two distributions $F_1$ and $F_2$
supported on the same compact metric space,
let $X_1,X_2,\dots,X_n$ be drawn i.i.d.\ from the mixture
$\lambda F_1 + (1-\lambda) F_2$,
and let 
$Y_i = 1$ if $X_i \sim F_1$ and $Y_i = 0$ otherwise, for $i \in [n]$.
Letting $I_1 = \{i: Y_i=1\}$ and $n_1 = |I_1|$
and defining $I_2$ and $n_2$ analogously, the MMD is given by
\begin{equation*}
M_n = 
\sum_{i,j \in I_1 \text{ distinct}}
		\frac{ \kappa(X_i,X_j) }{ n_1(n_1-1) }
        + \sum_{ i,j \in I_2 \text{ distinct}}
		\frac{ \kappa(X_i,X_j) }{ n_2(n_2-1)}
        - \sum_{i \in I_1} \sum_{j \in I_2}
                    \frac{\kappa(X_i, X_j) }{n_1 n_2} ,
\end{equation*}
where $\kappa$ is the kernel of a reproducing kernel Hilbert space.
By definition, $M_n$ is a U-statistic in
$X_1,X_2,\dots,X_n$ with kernel
$h((x_i,y_i),(x_j,y_j)) = (-1)^{y_i-y_j}\kappa(x_i,x_j)$.
Suppose that $F_1$ and $F_2$ are such that $F = \lambda F_1 + (1-\lambda)F_2$
is a $d$-dimensional inner product distribution,
and suppose that we observe
$(A,X) \sim \RDPG(F,n)$ along with the
indicators $Y_1,Y_2,\dots,Y_n$.
A natural approach to testing the hypothesis
\begin{equation*}
H_0 : F_1 = F_2
\end{equation*}
is to form the MMD test statistic from the (estimated) latent positions,
\begin{equation*}
\Mhat_n =
\sum_{i,j \in I_1 \text{ distinct}}
		\frac{ \kappa(\Xhat_i,\Xhat_j) }{ n_1(n_1-1) }
        + \sum_{ i,j \in I_2 \text{ distinct}}
		\frac{ \kappa(\Xhat_i,\Xhat_j) }{ n_2(n_2-1)}
        - \sum_{i \in I_1} \sum_{j \in I_2}
                    \frac{\kappa(\Xhat_i, \Xhat_j) }{n_1 n_2} ,
\end{equation*}
where $\Xhat = \ASE(A,d)$.
A slight variant of this statistic was considered by
\cite{TanAthSusLyzPri2017} for the purpose of two-network
hypothesis testing.
\end{example}

\begin{example}[Degree Moments]
\normalfont
The degree distribution of a graph carries important information about
graph structure.
Measures such as the variance of the degrees,
\begin{equation*}
V_d(A) = n^{-2} \sum_{i,j} \left(\frac{d_i - d_j}{n}\right)^2,
\end{equation*}
where $d_i = \sum_k A_{i,k}$ is the degree of the $i$-th vertex,
provide a useful summary of vertex behavior. 
Rearranging the sum, we have 
\begin{equation*}
\E[ V_d(A) \mid X]
= n^{-4} \sum_{i,j,k,\ell} \E\left[ A_{i,k}(A_{i,\ell} - A_{j,\ell}) \mid X \right]
= n^{-4} \sum_{i,j,k,\ell} X_i^TX_k( X_i - X_j)^T X_\ell,
\end{equation*}
which is a V-statistic \citep{Serfling1980} in the latent positions
after appropriate symmetrization.
Similar results can be shown for other central moments of the degree
distribution.   
\end{example}

A number of other network quantities are
expressible similarly, either under a different latent geometry
or after appropriate rescaling by
some network-dependent quantity that converges almost surely to
a parameter depending only on the latent position distribution.
Examples include
measures of assortative mixing by degrees \citep{Newman2010},
energy statistics \citep{SzeRiz2013,LeeShePriVog2017}
and Randi\'{c}'s connectivity index \citep{Randic1975}.



\begin{remark}[Specifying the Parameter of Interest]
Among the examples above, we have seen statistics of the form $t : \{0,1\}^{n \times n} \rightarrow \bbR$ (e.g., the average degree),
where $\E[ t(A) \mid X ]$ is expressible as a U-statistic in the latent positions.
We stress, however, that the target of inference here and in what follows is
$\E[ t(A) ]$, rather than the conditional expectation $\E[ t(A) \mid X ]$ or, say, $t( \E A)$.
Indeed, it is precisely because $\E[ t(A) \mid X]$ is expressible as a U-statistic in the latent positions that we are able to construct a confidence interval for $\E h(X_1,X_2,\dots,X_m) = \E t( A )$.
\end{remark}

\subsection{Consistency of Network U-statistic Bootstrap}

Having seen how U-statistics arise in the statistical analysis of networks,
we return to the setting where the latent positions $X_1,X_2,\dots,X_n$ are not available and only the matrix $A$ is observed.
Letting $\Xhat_1,\Xhat_2,\dots,\Xhat_n \in \bbR^d$ be the rows of
$\Xhat = \ASE(A,d)$, we consider the plug-in U-statistic,
\begin{equation} \label{eq:def:Uhat}
\Uhat_n = \binom{n}{m}^{-1} \sum_{1 \le i_1 < i_2 < \dots < i_m \le n}
        h(\Xhat_{i_1},\Xhat_{i_2},\dots,\Xhat_{i_m}).
\end{equation}
If this quantity is to resemble $U_n$, the kernel $h$ must be invariant to the non-identifiability inherent to the random dot product graph,
and thus we make the following assumption.
\begin{assumptionu} \label{assumu:hinvariant}
Let $\bbO_d$ denote the set of all $d$-by-$d$ orthogonal matrices.
The kernel $h$ satisfies
\begin{equation*} 
  h(X_1,X_2,\dots,X_n) = h(QX_1,QX_2,\dots,QX_n)~~\text{ for all }~~Q \in \bbO_d.
\end{equation*}
\end{assumptionu}
The main results of this section state that
for suitably smooth kernel functions,
the plug-in estimate in Equation~\eqref{eq:def:Uhat}
and bootstrap samples formed from it
are asymptotically equivalent to the U-statistic formed from the true
latent positions $X_1,X_2,\dots,X_n \iid F$.
The following assumption makes this notion of smoothness precise. 
\begin{assumptionu} \label{assumu:smooth}
Let $\nabla^2 h : \bbR^{md} \rightarrow \bbR^{md \times md}$
denote the Hessian of kernel $h$.  We assume that 
$\nabla^2 h$ is continuous on the closure of $\supp F$ and there exists
a neighborhood $\calS \subseteq (\bbR^d)^m$ of $\supp F$ satisfying
\begin{equation*} 
\sup \{ \| \nabla^2 h(Z_1,Z_2,\dots,Z_m) \| : (Z_1,Z_2,\dots,Z_m) \in \calS \} < \infty.
\end{equation*}
\end{assumptionu}

The following theorem shows that Assumptions~\ref{assumu:hinvariant} and~\ref{assumu:smooth} 
are sufficient to ensure that the plug-in
U-statistic $\Uhat_n$ recovers $U_n$ asymptotically.
The proof is given in Appendix~\ref{apx:inprob}.

\begin{theorem} \label{thm:ustat:consistent}
Let $F$ be a $d$-dimensional inner product distribution
and suppose that 
$h: (\bbR^d)^m \rightarrow \bbR$ is a
symmetric kernel satisfying Assumptions~\ref{assumu:hinvariant}
and~\ref{assumu:smooth}.
Suppose $(A,X) \sim \RDPG(F,n)$ 
and let $\Xhat = \ASE(A,d)$. 
Let $U_n$ and $\Uhat_n$ be the U-statistics based on, respectively,
the true latent positions $X$ 
and their ASE estimates $\Xhat$.
Then $\sqrt{n}(\Uhat_n - U_n) \rightarrow 0$ almost surely. %
\end{theorem}

From the fact that $U_n$ converges almost surely to
the population parameter $\theta = \theta(F) = \E U_n$
\citep[][Theorem 5.4 A]{Serfling1980},
$\Uhat_n$ is a strongly consistent estimate of $\theta$.
Further, by Slutsky's Theorem,
$\Uhat_n$ has the same distributional limit as $U_n$.
\begin{corollary}
Under the settings of Theorem \ref{thm:ustat:consistent},
the plug-in U-statistic $\Uhat_n$
satisfies $\Uhat_n \rightarrow \E U_n = \theta(F)$ almost surely
and $\sqrt{n}(\Uhat_n - \theta) \inlaw \calN(0,m^2 \zeta_1)$,
where $\zeta_1$ is defined in Equation~\eqref{eq:def:zeta}.
\end{corollary}

Our main goal, however, is not establishing convergence of $\Uhat_n$ but the the more delicate task 
of obtaining bootstrap samples
to approximate the sampling distribution of $U_n$.
If we knew the true values of the latent positions,
any number of techniques for bootstrapping U-statistics would work.  The idea is thus to construct, instead of a bootstrap sample $U_n^*$
as in~\eqref{eq:def:efron}, a plug-in version 
\begin{equation*} \label{eq:def:efronhat}
\Uhat^*_n
= \binom{n}{m}^{-1} \sum_{c \in \calC^n_m} 
	\bbW_{c}~
        h(\Xhat_{i_1},\Xhat_{i_2},\dots,\Xhat_{i_m}),
\end{equation*}
where again $\bbW \in \bbR^{\calC^n_m}$ is a vector of random weights
independent of the observed network
and $\Xhat_1,\Xhat_2,\dots,\Xhat_n \in \bbR^d$
are the latent position estimates.
We will assume that these are ASE estimates, 
but we stress that similar results can be obtained 
under any estimation scheme that recovers the latent positions
at a suitably fast rate.
As mentioned in Section~\ref{subsec:prelim:ustat},
the specific conditions on $\bbW$ needed to ensure
the distributional consistency of $\Uhat^*_n$ vary,
but for our plug-in scheme to work, 
we require the following growth condition.
\begin{assumptionw} \label{assumw:growth}
The weight vector $\bbW$ satisfies
\begin{equation*} 
\max_{c \in \calC^n_m} |\bbW_{c}|
= o\left( \frac{ \sqrt{n} }{ \log^2 n } \right).
\end{equation*}
\end{assumptionw}
With this assumption, we have the following theorem for the plug-in U-statistic  bootstrap.  The proof is similar to that of Theorem~\ref{thm:ustat:consistent} and is given in Appendix~\ref{apx:inprob}.
\begin{theorem} \label{thm:Uboot}
Let $F$ be a $d$-dimensional inner product distribution
and suppose that 
$h : (\bbR^d)^m \rightarrow \bbR$ is a symmetric kernel
satisfying Assumptions~\ref{assumu:hinvariant} and~\ref{assumu:smooth}.
Let $(A,X) \sim \RDPG(F,n)$ and  $\Xhat = \ASE(A,d)$.  Let $\zeta_1$ be as defined in Equation~\eqref{eq:def:zeta}.  Then, 
\begin{enumerate}
    \item The ASE plug-in bootstrap
\begin{equation*}
\UBFhat^*_n = \binom{n}{m}^{-1} \sum_{c \in \calC^n_m}
        h(\Xhat_{i_1}^*,\Xhat_{i_2}^*,\dots,\Xhat_{i_m}^*)
\end{equation*}
satisfies
\begin{equation*}
\sqrt{n}(\UBF^*_n - \theta) \inlaw \calN(0,m^2 \zeta_1).
\end{equation*}

\item Let $U^*_n$ be the weighted bootstrapped U-statistic
defined in Equation~\eqref{eq:def:efron} and let $\Uhat^*_n$ be its plug-in version. 
If $U^*_n$ is distributionally consistent
and the weight vector $\bbW \in \bbR^{\calC^n_m}$
satisfies Assumption~\ref{assumw:growth}, then
\begin{equation*}
\sqrt{n}(\Uhat^*_n-\theta) \inlaw \calN(0,m^2 \zeta_1) .
\end{equation*}
\end{enumerate}
\end{theorem}

\begin{remark}[The degenerate case]
The reader familiar with U-statistics may wonder what can be said in the
more challenging setting where the kernel $h$ is degenerate with respect
to $F$ \citep[][Chapter 5]{Serfling1980}.
It is known that if $h$ is $r$-degenerate, then
$n^{r/2}(U_n - \theta)$ converges to a nondegenerate limiting distribution,
which is not, in general, normal
\citep[][]{Serfling1980,ArcGin1992}.   A result analogous to Theorem~\ref{thm:Uboot}
for this case, unfortunately, does not appear feasible, since the concentration of
$U_n$ about the true parameter $\theta$ is of
a smaller order than the best known concentration of the estimates
$\Xhat_1,\Xhat_2,\dots,\Xhat_n$ about the true latent positions
\citep{LyzSusTanAthPri2014}.
That is, the estimation error in Lemma~\ref{lem:2toinfty} does not
vanish fast enough to yield convergence of $\Uhat^*_n$ to $U^*_n$
in probability.
\end{remark}

\subsection{Computational concerns}

Both U-statistics and the bootstrap are well known to be
computationally intensive. As a result, a na\"ive implementation of our
positional U-statistic resampling scheme would be of little practical utility for large $n$.
The following additive weighted bootstrapping procedure
alleviates these computational expenses.
This procedure
discussed at length in Chapter 4 of \cite{BosCha2018},
whose presentation we follow.

Consider a U-statistic with kernel $h$ taking $m$ arguments.
Having generated a vector of weights $W \in \bbR^n$,
we form the weight vector $\bbW \in \bbR^{\calC^n_m}$ by setting
$\bbW_c = \sum_{k=1}^m W_{i_k}/m$ for each
$c = (i_1,i_2,\dots,i_m) \in \calC^n_m$.
While a number of choices for the distribution of $W$ are possible
\citep[see][for discussion]{HusJan1993,BosCha2018},
we take $W \sim \Multinomial(n,n^{-1})$ in our experiments
in Section~\ref{sec:expts} for simplicity.
A concentration inequality applied entry-wise to $W$
followed by a union bound is enough to ensure that
$\max_{c \in \calC^n_m} | \bbW_c | \le C \log^m n$,
so that Assumption~\ref{assumw:growth} is satisfied.
The additive structure of $\bbW$ enables a useful computational speedup
in computing bootstrap replicates of $U_n$.
We construct, for each $i \in [n]$, the quantity
\begin{equation} \label{eq:def:BC:Utilde}
 \Utilde_{ni} = \binom{n-1}{m-1}^{-1}
                \sum_{c \in \calC^n_m ~: ~i \in c}
                h(X_{i_1},X_{i_2},\dots,X_{i_m}).
\end{equation}
Recalling our definition of $U^*_n$ from
Equation~\eqref{eq:def:efron}, it is simple to verify that
\begin{equation*}
 U^*_n = n^{-1} \sum_{i=1}^n W_i \Utilde_{ni}.
\end{equation*}
As discussed in \cite{BosCha2018}, this enables generation of
many bootstrap samples after only a single instance of
$O(n^m)$ computation time to construct the $\Utilde_{ni}$,
rather than $O(Bn^m)$ computation time to generate $B$ bootstrap
samples under a more na\"ive implementation.

Unfortunately, the $O(n^m)$ time required to compute
the quantity in Equation~\eqref{eq:def:BC:Utilde} for all $i \in [n]$
may still be quite substantial, particularly if $m$ is larger than $2$.
We can further reduce the computational cost by
making use of incomplete U-statistics \citep{Blom1976,CheKat2019},
replacing the average in~\eqref{eq:def:BC:Utilde}
with a Monte Carlo estimate, drawing for each $i \in [n]$
a uniform random sample of size $M$ with replacement from
the set $\{ c \in \calC^n_m : i \in c \}$.
With this modification, our method for 
obtaining $B$ bootstrap samples for a U-statistic in the latent 
positions requires a single low-rank spectral decomposition
followed by $O( (M+B)n )$ sampling operations.
Thus our bootstrap is far less computationally demanding than existing algorithms for generating bootstrap samples of subgraph counts, which
require extensive sampling and counting operations.

\subsection{Sparse networks}
For notational simplicity, our results above were written under the assumption that the latent position distribution $F$ does not depend on $n$. This results in the expected node degrees growing linearly in $n$, which is unrealistic in many applications.   A natural way to allow for sparser networks under the RDPG is to keep $F$ fixed and rescale the expectation of $A$ by a sparsity factor $\rho_n \rightarrow 0$,
so that $\E[ A \mid X ] = \rho_n X X^T$.
For identifiability, we assume that if $X_1,X_2$ are independent draws from $F$,
then $\E X_1^T X_2 = 1$, and for all suitably large $n$,
\begin{equation} \label{eq:assum:sparsity}
\forall x,y \in \supp F  : 0 \le \rho_n x^T y \le 1.
\end{equation}
Analogous assumptions are made in \cite{BicCheLev2011,BhaBic2015,GreSha2017}.
We may equivalently think of rescaling all latent positions by $\sqrt{\rho_n}$. 
Under this scaling, the ASE plug-in U-statistic $\Uhat_n$ now estimates
$\E h(\sqrt{\rho_n}X_1,\dots,\sqrt{\rho_n}X_m)$
rather than $\E h(X_1,\dots,X_m)$, and thus we must specify how
the kernel $h$ behaves with respect to scaling of its arguments.
A consistency result analogous to Theorem~\ref{thm:ustat:consistent}
can be established under this setting
with an additional homogeneity assumption on $h$.

\begin{theorem} \label{thm:sparse:ustat:consistent}
Under the sparse setting just described,
let $h$ be a kernel satisfying Assumptions~\ref{assumu:hinvariant}
and~\ref{assumu:smooth}, with $\supp F$ replaced with the convex hull of $\{ 0 \} \cup \supp F$.
Suppose in addition that there exists $r \ge 1$ such that
for all $\alpha \ge 0$ and all $x_1,x_2,\dots,x_m \in \supp F$,
\begin{equation} \label{eq:homogeneous}
h(\alpha x_1, \dots, \alpha x_m )
        = \alpha^r h(x_1, \dots,x_m).
\end{equation}
Define the estimator $\rhohat_n = \binom{n}{2}^{-1} \sum_{i<j} A_{ij}$.
Then
\begin{equation*}
\frac{ \sqrt{n}(\Uhat_n - U_n) }{ \rhohat_n^r } \rightarrow 0.
\end{equation*}
\end{theorem}
\begin{proof}
The result follows from the strong law of large numbers, 
once one establishes convergence of $\ASE(A,d)$ to $\sqrt{\rho_n}X$.
Details are provided in Appendix~\ref{apx:sparse}.
\end{proof}
The condition in Equation~\eqref{eq:homogeneous} is satisfied by
the average degree, subgraph count and degree moment examples presented at the beginning of this section.
Whether or not the MMD obeys this
condition depends on the MMD kernel $\kappa$.

At first glance, one might hope to obtain next a distributional result
for $\Uhat_n$ about $\rho_n^r \theta = \E U_n$ by 
using the delta method to establish a distributional limit for
\begin{equation*}
\frac{ \sqrt{n}( U_n - \rho_n^r \theta) }{ \rhohat_n^r }
=
\sqrt{n} \frac{
	\sum_{1 \le i_1 < \dots < i_m \le n}
	(h(X_{i_1}, X_{i_2}, \dots, X_{i_m}) - \theta) }
	{ \rhohat_n^r/\rho_n^r }
\end{equation*}
and appealing to Slutsky's Theorem.
Unfortunately, the estimator $\rhohat_n$ complicates matters.
Applying the delta method would require that we control
the covariance term relating
$\rhohat_n^r/\rho_n^r$ and $\sum_c h(X_{i_1},\dots,X_{i_m})$,
which depends on the kernel $h$.
In the case of subgraph counts, this covariance can be controlled
so long as the subgraph $R$ has a particular structure.
Such structural assumptions on $R$ are required in the distributional results
presented by \cite{BhaBic2015} and \cite{GreSha2017}.

Focusing on subgraph counts, for a graph $R$ on $m$ vertices with edge set $E = E(R)$, we have the kernel
\begin{equation*}
h_R(x_1,x_2,\dots,x_m)
= \frac{1}{N(R)} \sum_{\tau \in S_m} 
        \prod_{\{i,j\} \in E} x_{\tau(i)}^T x_{\tau(j)}
        \prod_{\{i,j\} \not \in E} (1-x_{\tau(i)}^T x_{\tau(j)}),
\end{equation*}
where $N(R)$ is the number of graphs isomorphic to $R$.
Thus,
\begin{equation*}
h(\sqrt{\rho_n}x_1,\sqrt{\rho_n}x_2,\dots,\sqrt{\rho_n}x_m)
        = \rho_n^{|E|}(1-\rho_n)^{\binom{m}{2}-|E|} h(x_1,x_2,\dots,x_m).
\end{equation*}
Since $\rho_n \rightarrow 0$, we must rescale by $\rho_n^{-|E|}$, which  yields the {\em normalized} subgraph density
considered in~\cite{BicCheLev2011},
\begin{equation*}
 \Ptilde(R) = \rho_n^{-|E|} \E h_R(\sqrt{\rho_n} X_1, \sqrt{\rho_n} X_2, \dots,
				\sqrt{\rho_n} X_m )
\end{equation*}
Bootstrapping this quantity is
the focus of both \cite{BhaBic2015} and \cite{GreSha2017}.
Note that estimating $\Ptilde(R)$ requires estimating both
$\rho_n$ and the subgraph density
$\E h_R(\sqrt{\rho_n} X_1, \sqrt{\rho_n} X_2, \dots, \sqrt{\rho_n} X_m )$.
Letting $\rhohat_n = \binom{n}{2}^{-1} \sum_{i<j} A_{ij}$ again,
we have the following distributional result.

\begin{theorem} \label{thm:sparse:ustat:clt}
Under the sparse setting described above,
let $R$ be a graph on $m$ vertices, either acyclic or equal to
a cycle on $m$ vertices.
Provided that $n\rho_n  = \omega( \log n )$,
\begin{equation*}
\sqrt{n}\, \frac{ \Uhat_n - \rho_n^r \E h(X_1,X_2,\dots,X_m) }
		{ \rhohat^r }
\inlaw \calN(0, \sigma^2),
\end{equation*}
where $\sigma^2$ depends on $R$ and $F$.
The same result holds for the ASE plug-in bootstrap $\Uhat_n^*$.
\end{theorem}
\begin{proof}
The result follows by Theorem~\ref{thm:sparse:ustat:consistent}
followed by a delta method argument applied to the ratio of
$U_n - \rho_n^r \E h(X_1,X_2,\dots,X_m)$ and $\rhohat_n/\rho_n$.
The delta method argument is fairly similar to that
of Theorem 1 in \cite{BicCheLev2011} and is thus omitted.
\end{proof}

\section{Generating Network Bootstrap Samples}
\label{sec:netboot}

In this section, we turn to the more general case of bootstrapping  network quantities that cannot be expressed as U-statistics in the latent positions.  We start with two examples of such quantities.  

\begin{example}[Global Clustering Coefficient] \normalfont
\label{ex:globalclusteringcoef}
The global clustering coefficient
measures the total fraction of all ``open''
triangles that are closed, counting how many of all vertex triples $i,j,k$
for which $A_{ij} = A_{jk} = 1$ also have $A_{jk} = 1$.  Thus for a graph $G$ with adjacency matrix $A$, the global clustering coefficient is given by 
\begin{equation*}
C(G) = \frac{ 3 \sum_{1 \le i < j < k \le n} A_{ij} A_{jk} A_{ki}  }
        {\sum_{1 \le i < j < k \le n}
                (A_{ij}A_{jk} + A_{jk}A_{ki} + A_{ki}A_{ij} ) }.
\end{equation*}
Letting $L_3$ denote the linear chain on three vertices and $K_3$ denote the
triangle on $3$ vertices, we can write $C(G)$ as a ratio of subgraph
counts, 
\begin{equation*}
C(G) = \frac{ F_{K_3}(G) }{ F_{L_3}(G)} .
\end{equation*}
A quantity similar to this appeared in \cite{BhaBic2015}, who construct a confidence interval via the delta method and properties of subgraph counts.
\end{example}

\begin{example}[Average Shortest Path Length] \normalfont
\label{ex:ASPL}
For a graph $G$ on $n$ vertices with adjacency matrix $A$,
the shortest path distance
between vertices $i$ and $j$, which we denote $d_A(i,j)$,
is given by the length of the shortest path connecting vertices $i$
and $j$ in $G$. We take $d_A(i,j) = \infty$ if $i$ and $j$ are in
different connected components.
The average shortest path distance,
\begin{equation} \label{eq:def:spdist}
\dbar_A = \binom{n}{2}^{-1} \sum_{i < j} d_A(i,j)
\end{equation}
provides a natural measure of the extent to which the graph
exhibits small-world behavior.
While at first glance Equation~\eqref{eq:def:spdist} looks like a U-statistic,
this is not the case, since $d_A$ is a function that
itself depends on the data, rather than being a fixed kernel.
As a result, our U-statistic results above do not apply.
\end{example}

For quantities such as these that cannot be directly represented as a U-statistic, 
a natural approach to generating bootstrap samples would be to repeatedly generate
random adjacency matrices $\Ahat^*$ having similar distribution to $A$
and compute the network statistic on those replicates.
For this approach to work, we need, at a minimum, that the bootstrapped network $\Ahat^*$ be similar in distribution to $A$,
which in turn requires a notion of distance on networks.
A well-known example of such a distance is the cut metric
\citep{Lovasz2012}, which is especially well-suited to node-exchangeable
random graphs because it metrizes convergence of subgraph densities
\citep[][Theorem 11.3]{Lovasz2012}.
A resulting drawback of the cut metric is that it captures only the local information of small subgraphs, and fails to capture larger network structures.
A step toward a more global network distance is given by the
graph matching distance,
which measures the fraction of edges that differ
between two graphs, after their vertices have been aligned
so as to minimize the number of such edge discrepancies.

\begin{definition}[Graph matching distance] \label{def:dGM}
Let $G_1$, $G_2$ be two graphs each on $n$ vertices,
with adjacency matrices $A_1, A_2 \in \bbR^{n \times n}$.
The graph matching distance is defined as 
\begin{equation} \label{eq:def:dGM}
\dGM(A_1,A_2) 
= \min_{P \in \Pi_n} \binom{n}{2}^{-1} \frac{ \| A_1 - P A_2 P^T \|_1 }{ 2 },
\end{equation}
where $\Pi_n$ denotes the set of all $n$-by-$n$ permutation matrices.
\end{definition}
Two sequences of networks converge in this distance if they
are, asymptotically, isomorphic to one another
up to a vanishing fraction of their edges.
\begin{observation}
The graph matching distance $\dGM$ is a distance.
\end{observation}
\begin{proof}
Symmetry and non-negativity follow from the definition.
The triangle inequality can be
verified by noting that, 
letting $H$ be the adjacency matrix of another $n$-vertex graph,
\begin{equation*} \begin{aligned}
\min_{P \in \Pi_n} \| A_1 - P A_2 P^T \|_1
&= \min_{P,Q \in \Pi_n} \| A_1 - Q H Q^T + Q H Q^T - P A_2 P^T \|_1 \\
&\le \min_{Q \in \Pi_n} \| A_1 - Q H Q^T \|_1
	+ \min_{P,Q \in \Pi_n} \| Q H Q^T - P A_2 P^T \|_1
\end{aligned} \end{equation*}
and observing that minimization over $P,Q \in \Pi_n$ is equivalent to
minimizing $\| H - P A_2 P^T \|_1$ with respect to $P \in \Pi_n$.
\end{proof}
We note that this distance has appeared in the literature under
a number of different names.
See \cite{BenIoa2018} and citations therein.

The optimization in Equation~\eqref{eq:def:dGM} is a quadratic assignment
problem \citep{BurAmiMar2009}, and thus computationally hard 
in general.   However, the graph matching problem has been studied extensively
\citep[see][and citations therein]{ConFogSanVen2004,Lyzinski2018},
and fast approximate solvers exist \citep{FAQAP}, though we will not
need them here.
The graph matching distance is,
up to constant factors depending on choice of normalization,
an upper bound on the cut metric,
which is also based on a computationally hard optimization problem.
This upper bound is immediate from the fact that the 
cut norm of a matrix $M$ is upper bounded by $\|M\|_1/n^2$
\citep[see][Chapter 8]{Lovasz2012}.

With this network distance in hand, we define a Wasserstein
distance between graphs analogously to the well-known
Wasserstein distance between Euclidean random variables.

\begin{definition}
Let $p \ge 1$ and let $A_1$ and $A_2$ be the adjacency matrices of
two random graphs both on $n$ vertices,
and let $\Gamma(A_1, A_2)$ denote the set of all couplings of $A_1$ and $A_2$.
The Wasserstein $p$-distance between $A_1$ and $A_2$ is given by
\begin{equation*}
\WGM{p}^p(A_1,A_2) = \inf_{\nu \in \Gamma(A_1,A_2)}
	\int \dGM^p(A_1,A_2) d \nu.
\end{equation*}
\end{definition}

The following lemma shows that, essentially,  the Wasserstein distance between two
random dot product graphs is bounded by the Wasserstein distance between their
respective latent position distributions (up to an orthogonal transformation). Since the latter Wasserstein distance must account for the orthogonal rotation,  
we define the Wasserstein $p$-distance between two $d$-dimensional inner product distributions $F_1$, $F_2$, 
for all $p \ge 1$, as 
\begin{equation} \label{eq:def:dcirc}
\dcirc_p(F_1,F_2) = \min_{Q \in \bbO_d} d_p(F_1,F_2 \circ Q).
\end{equation}
The lemma below will be the main technical tool required to show that the
bootstrapped $\Ahat^*$ described above converges to $A$ in
the graph matching Wasserstein distance.
We note that, after this paper was written, we found that a similar result was shown by \cite{Lei2018graphs}, in a different context and 
for the weaker cut metric instead of the graph matching distance.
A proof of the lemma can be found in Appendix~\ref{apx:wasser}.

\begin{lemma} \label{lem:GMLPwasser}
Let $F_1, F_2$ be $d$-dimensional inner product distributions
with $A_1 \sim \RDPG(F_1,n)$ and $A_2 \sim \RDPG(F_2,n)$. Then
\begin{equation*}
\WGM{p}^p(A_1,A_2) \le 2 \dcirc_1 (F_1,F_2)
\end{equation*}
\end{lemma}
Recall the procedure for generating RDPG bootstrap replicates,
outlined in Section~\ref{subsec:intro:netboot}.
Given $A$, we obtain estimates of the latent positions
$\Xhat_1,\Xhat_2,\dots,\Xhat_n$ via the ASE applied to $A$.
Letting $\Fhat_n$ be the empirical distribution of the estimates,
we draw $(\Ahat^*,\Xhat^*) \sim \RDPG(\Fhat_n,n)$.
The convergence rate of the ASE ensures that $\Fhat_n$ approximates $F_n$
well (up to orthogonal transformation),
and the fact that the empirical distribution approximates the population
distribution ensures that $F_n$ is close to $F$.
Lemma~\ref{lem:GMLPwasser} thus suggests that
$\Ahat^*$ will be distributionally similar to $A$.
The following theorem makes this precise.
A proof can be found in Appendix~\ref{apx:wasser}.

\begin{theorem} \label{thm:RDPGboot:wasser}
Let $F$ be a $d$-dimensional inner product distribution with $(A,X) \sim \RDPG(F,n)$.
Letting $\Fhat_n$ denote the empirical distribution of the ASE estimates $\Xhat$,
generate $\Ahat^* \sim \RDPG(\Fhat_n,n)$.
Then, if $(H,Z) \sim \RDPG(F,n)$ is an independent copy of $(A,X)$,
\begin{equation*}
\WGM{p}^p(\Ahat^*,H) = O\left( (n^{-1/2} + n^{-1/d})\log n \right) \, .
\end{equation*}
\end{theorem}

Since the graph matching distance is an upper bound on the cut metric,
which in turn metrizes convergence of subgraph densities, 
Theorem~\ref{thm:RDPGboot:wasser} implies that
the subgraph densities of $\Ahat^*$ converge almost surely to the same
limit as those of $H$.
We would also like that these counts
have the same distributional limits after appropriate rescaling, but 
unfortunately, this distributional limit is somewhat more delicate.
While Theorem~\ref{thm:RDPGboot:wasser} ensures
that our bootstrapped network replicate $\Ahat^*$ converges in the
Wasserstein metric
to the distribution of the original observation $A$,
this is not in itself sufficient to ensure that some network
statistic of interest, say $t(\Ahat^*)$,
converges to the same distribution as $t(H)$.
Provided $\E t(H)$ is finite,
it is sufficient that the network statistic in question be
continuous with respect to our network Wasserstein metric,
in the sense that
$\WGM{1}(\Ahat^*,H) \rightarrow 0$ implies
that (by slight abuse of notation), $d_1(t(\Ahat^*), t(H)) \rightarrow 0$.
Proving this continuity even for (suitably rescaled) subgraph counts
does not appear feasible using the techniques underlying
Theorem~\ref{thm:RDPGboot:wasser},
as the coupling argument used in the proof
fails in the presence of the $\sqrt{n}$ scaling needed to
ensure a non-degenerate limit for subgraph densities
\citep[see, e.g.,][Theorem 1]{BicCheLev2011}.
Such distributional results are even less clear for more complicated
network statistics. We thus leave this line of inquiry for future work.

\section{Experiments} \label{sec:expts}

In this section, we briefly demonstrate empirical performance of the methods introduced in
Sections~\ref{sec:ustat} and~\ref{sec:netboot} on two specific network statistic examples.
We begin with an application of the U-statistic bootstrap to the problem of estimating the triangle density.

\subsection{U-statistic bootstrap: triangle counts}
\label{subsec:expt:uboot}

As an application of the U-statistic-based bootstrapping method 
introduced in Section~\ref{sec:ustat},
we consider the problem of obtaining a confidence interval for the
triangle subgraph density $P(K_3)$ in a random dot product graph.
We note that this differs slightly from 
\cite{BhaBic2015} and \cite{GreSha2017}, who 
focus on the normalized density, $\Ptilde(K_3)$.
As shown in \cite{BicCheLev2011,BhaBic2015}
and discussed briefly
at the end of Section~\ref{sec:ustat},
the limiting distribution for $\Ptilde(K_3)$
can be obtained via the delta method.
The variance of this limit distribution is expressible in terms of
other subgraph densities, and thus can itself
be estimated via our U-statistic-based bootstrapping scheme,
but this is notationally cumbersome.
Thus, since the aim of this experiment is merely illustrative,
we bootstrap the simpler $P(K_3)$, while stressing that
obtaining a confidence interval for $\Ptilde(K_3)$ instead 
is straightforward, and refer the interested reader to Section 3 of \cite{BhaBic2015} for details.  

For the experiments on synthetic data, we generate networks from the RDPG with latent position distribution $\opBeta(2,3)$, as  $(A, X) \sim \RDPG( \opBeta(2,3), n )$ for various choices of $n$, 
and construct the latent position estimate $\Xhat = \ASE(A,1)$.
While the diagonal entries of $A$ are negligible for the purposes of
asymptotics, the zeros on the diagonal of $A$ tend to introduce instability to the ASE for finite $n$.
Several methods for correcting this have been suggested
\citep{SchTuc2010,MarPriCop2011,TanKetBadCalMarVogPriSus2019}.
Here we follow the simple approach of taking the $i$-th diagonal entry of $A$
to be the normalized degree of the $i$-th vertex,
i.e., $A_{ii} = n^{-1} \sum_{j\neq i} A_{ij}$.

Having obtained an estimate $\Xhat = \ASE(A,1)$, we generate
$100$ bootstrap samples of the triangle density from $\Xhat$,
as described in Section~\ref{sec:ustat}.
Based on this bootstrap sample, we produce $95\%$ confidence intervals
for the triangle density using three different methods.
The first is the percentile bootstrap, i.e.,
based on the empirical distribution of the bootstrap sample itself.
The other two confidence intervals are based on a normal approximation,
with variance estimated based on the bootstrap sample,
i.e., the standard bootstrap.
We consider confidence intervals of this form centered at the mean of the bootstrap sample and at the triangle density of the observed network $A$.
We include these three different confidence interval constructions to
assess the presence of bias in the bootstrap sample and to investigate
how well the bootstrap distribution approximates the true distribution
of the triangle density.

\begin{figure}[ht]
  \centering
  \subfloat[]{\includegraphics[width=0.5\columnwidth]{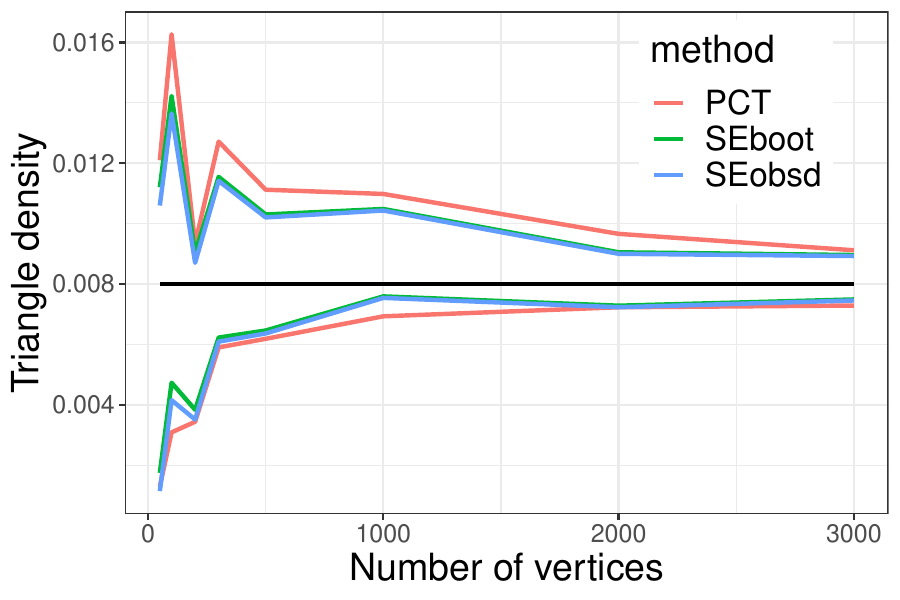}}
~~~
  \subfloat[]{\includegraphics[width=0.5\columnwidth]{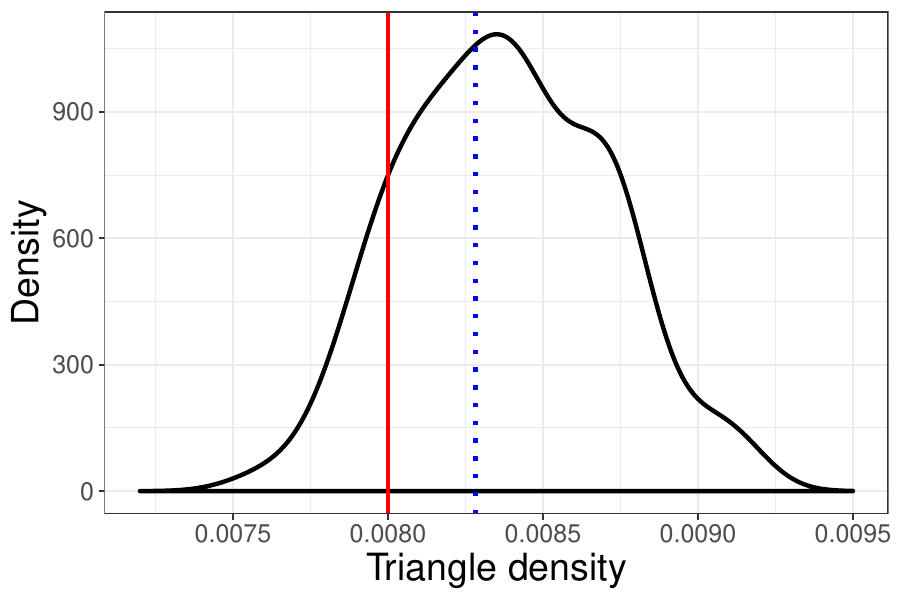}}
  \vspace{-3mm}
  \caption{(a) Confidence intervals produced by a single run of the triangle density experiment for varying values of $n$. The plot shows the confidence intervals produced by the percentile bootstrap (PCT, red), the standard bootstrap centered at the mean of the bootstrap samples (SEboot, green) and the standard bootstrap centered at the triangle density of the observed network (SEobsd, blue). (b) Smoothed density plot of the bootstrap samples produced in a particular experiment trial. The plot shows the same samples as were used to compute the $n=500$ confidence interval in subplot (a). The red solid line indicates the true expected triangle density. The blue dotted line indicates the triangle density of the network from which the bootstrap iterates were produced.}
    \label{fig:triden:singlerun}
\end{figure}

Figure~\ref{fig:triden:singlerun}
shows the bootstrap samples generated by a single run of this experiment
and the resulting confidence intervals.
Unsurprisingly, the confidence intervals produced by the percentile bootstrap
are slightly wider than those produced based on the normal approximation.  
In a smaller set of experiments, we found that this gap shrank,
but did not disappear entirely, when
the number of bootstrap samples was increased by an order of magnitude.
The trends in Figure~\ref{fig:triden:singlerun} are borne out
in Figure~\ref{fig:triden:summary},
which summarizes the performance of these confidence intervals,
aggregated over 200 independent realizations 
for each value of the number of vertices $n$.
Figure~\ref{fig:triden:summary}(a) confirms that the percentile bootstrap
confidence intervals are wider, on average,
than the standard bootstrap intervals for this problem.
Note that by construction, the two variants on the standard bootstrap have
the same length, and thus their lines in the plot overlap.
Figure~\ref{fig:triden:summary}(b) shows coverage rates for the three bootstrap
variants, and confirms that 
both of the standard bootstrap variants attain
approximately 95\% coverage,
while the percentile bootstrap is somewhat conservative.

\begin{figure}[ht]
  \centering
  \subfloat[]{\includegraphics[width=0.5\columnwidth]{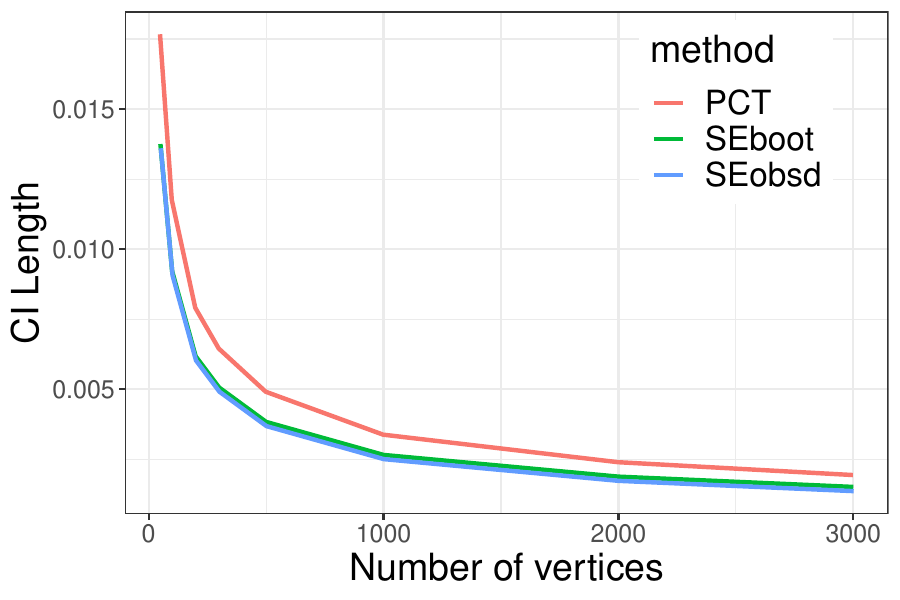}}
~~~
  \subfloat[]{\includegraphics[width=0.5\columnwidth]{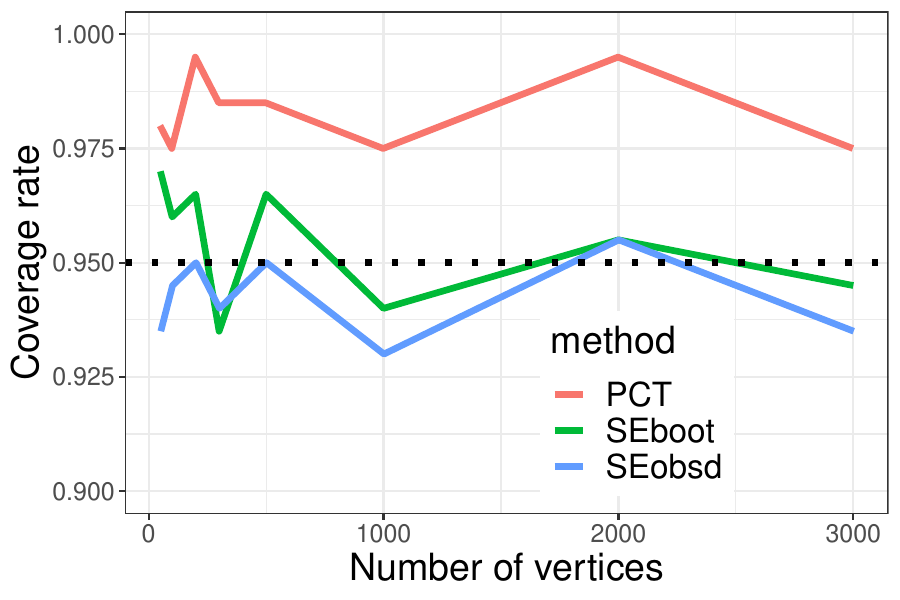}}
  \vspace{-3mm}
  \caption{Average performance of the percentile bootstrap (PCT, red), the standard bootstrap centered at the mean of the bootstrap samples (SEboot, green) and at the triangle density of the observed network (SEobsd, blue), aggregated over 200 trials, as a function of the number of vertices $n$. (a) Average confidence interval lengths.  By construction, the two standard bootstrap variants have the same length. (b) Average coverage rates. }
    \label{fig:triden:summary}
\end{figure}

The computational cost of the subgraph count bootstrap methods of \cite{BhaBic2015,GreSha2017} precluded a thorough comparison on this problem,
but a series of small-scale experiments suggest that both are
broadly competitive with our method, with comparable coverage rates and lengths to the ones produced by the normal approximations, but at a much higher computational cost.   We leave a more thorough comparison of the practical effectiveness of these methods for future work.

\subsection{Bootstrapping Average Shortest Path Length}
\label{subsec:expt:netboot}
We illustrate the full-network resampling scheme
discussed in Section~\ref{sec:netboot} on the problem of estimating the expected average shortest path length in a graph,
\begin{equation*}
\E[ \dbar_A \mid \dbar_A < \infty].
\end{equation*}
Note that we are conditioning on the event that the graph is connected to avoid the trivial situation where $\E[ \dbar_A ]$ is infinite for finite $n$.

The most natural approach to generating bootstrap samples of this quantity
is to generate bootstrap replicates of whole networks and evaluate the average shortest path length of each replicate.
The method introduced in Section~\ref{sec:netboot} is well-suited to this.

For comparison, we consider two other methods for generating network
bootstrap samples.
The first is an adaptation of the empirical graphon method 
introduced by \cite{GreSha2017}.
In that paper, the authors considered producing bootstrap samples for subgraph
counts by, in essence, resampling vertices.
That is, one produces a bootstrap replicate adjacency matrix
$A^*$ by drawing $Z_1,Z_2,\dots,Z_n$ i.i.d.\ from the uniform distribution
on $[n]$, and take $A^*_{ij} = A_{Z_i,Z_j}$.
As acknowledged by the authors, a major drawback of this scheme
is that since the diagonal elements of $A$ are equal to $0$,
resampled vertices $k,\ell$ with $Z_k=Z_\ell$ are precluded
from forming an edge.
For the purposes of estimating subgraph counts, this
drawback can be avoided, as demonstrated by the results in
\cite{GreSha2017}, but we will see that this causes a bias
when generating whole networks.
We expect that correcting for this deficiency
is possible, but it is not trivial and we do not pursue the question here.
We also include, for the sake of comparison, a parametric bootstrap procedure,
which performs estimation over a much smaller space of models
compared to the RDPG-based resampling scheme and the empirical graphon,
and can thus serve as a gold standard when its underlying model is true.

In this set of experiments,
we again generate $A$ from a random dot product graph with latent position
distribution given by $\opBeta(2,3)$. 
We discard and regenerate $A$ in the event that the resulting graph
is not connected, since we are interested in the average shortest path
in $A$ conditional on that average being finite.
Given $A$, we construct the estimate $\Xhat = \ASE(A,1)$,
replacing the zeros on the diagonal of $A$ with the scaled degrees
$n^{-1} \sum_j A_{ij}$ as before.
Letting $\Fhat_n$ denote the empirical distribution of $\Xhat$,
we then draw $B=100$ bootstrap replicates of
$(\Ahat^*,\Xhat^*) \sim \RDPG( \Fhat_n, n )$,
computing the average shortest path of each iterate
(and resampling in the event that a sample $\Ahat^*$ is not connected).
We note that for finite $n$, the ASE may produce an estimate $\Xhat$
such that some entries of $\Xhat \Xhat^T$ are outside the interval $[0,1]$.
When this occurs, we threshold the resulting
entries of the expected adjacency matrix so that
$\E[ \Ahat_{ij} \mid \Xhat ] = 0 \vee (1 \wedge \Xhat_i^T \Xhat_j)$.
For the empirical graphon procedure described above,
we resample from the same original observed network $A$.
For the parametric bootstrap, we fit the parameters of a beta distribution
to the entries of $\Xhat \in \bbR^n$ based on a method of moments estimate,
replacing $\Xhat$ with $-\Xhat$ in the event that the majority of the
entries of $\Xhat$ are negative, and discarding elements of $\Xhat$
that do not fall in $[0,1]$ after this adjustment.
Letting $(\ahat,\bhat)$ denote the estimated parameters of the
Beta distribution, we draw the $n \times n$ adjacency matrix $\Ahat^*$ from
$\RDPG( \opBeta(\ahat, \bhat), n)$.
In all three bootstrap procedures, we generate $100$ bootstrap samples,
discarding and regenerating samples as needed to ensure connected networks,
and use these $100$ bootstrap samples to obtain an estimate of the variance
of the average shortest path. We then use this variance to construct
a $95\%$ confidence interval centered on the observed
value of $\dbar_A$ via a normal approximation.

\begin{figure}[ht]
  \centering
  \subfloat[]{\includegraphics[width=0.5\columnwidth]{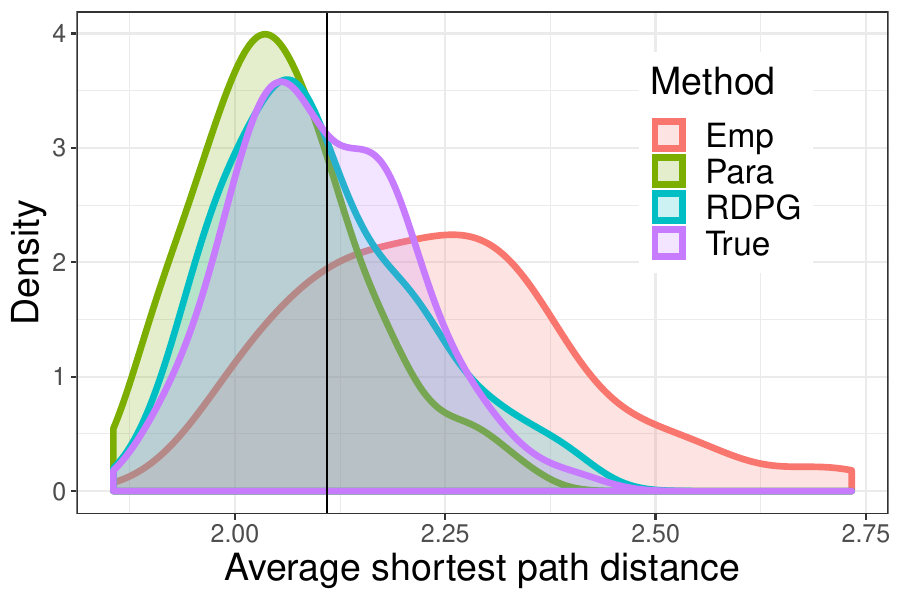}}
~~~
  \subfloat[]{\includegraphics[width=0.5\columnwidth]{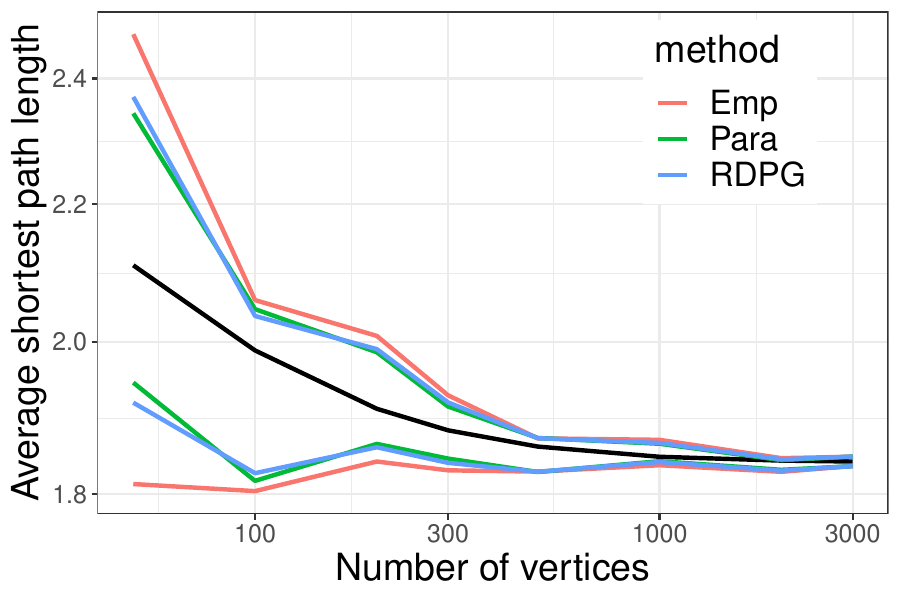}}
  \vspace{-3mm}
  \caption{(a) Bootstrap distributions produced by the different resampling schemes in a single trial with $n=50$.
The plot shows smoothed density plots of $100$ samples of the average shortest path distance generated by the empirical graphon (Emp, red),
RDPG resampling (RDPG, blue), and the parametric bootstrap (Para, green).
For comparison, a sample of $100$ draws from the true model is also shown (True, purple).  
The black vertical line shows the estimated mean of the true distribution, based on 10,000 Monte Carlo samples. 
(b) Confidence intervals produced based on the variances of the three different bootstrap samples for different values of $n$.
}
\label{fig:SPRDPG:singlerun}
\end{figure}

For each of $n=50,100,200,500,1000$, we ran $200$ independent trials
of the experiment just described.
For illustration, results from a single run of this experiment
are shown in Figure~\ref{fig:SPRDPG:singlerun}.
Figure~\ref{fig:SPRDPG:singlerun}(a) shows a smoothed density plot of the
$100$ bootstrap samples generated by the RDPG resampling scheme,
the empirical graphon and the parametric network bootstrap
for a single trial of the $n=50$ case.
In addition, the plot shows the histogram of $100$ independent draws from the
true distribution of the average shortest path under the generating model.
The black vertical line indicates the mean of this distribution,
estimated from 10,000 Monte Carlo samples, generated independently of the experimental trials.
It is clear from the plot that none of the three bootstrap methods captures
the true sampling distribution perfectly,
but the parametric and RDPG bootstraps both yield much better approximations
than the empirical graphon, which displays a significant positive bias.
Figure~\ref{fig:SPRDPG:singlerun}(b) shows representative confidence intervals
produced by the three sampling methods for each choice of $n$,
with the $n=50$ condition corresponding to the same trial as in (a).
We see that the RDPG and parametric bootstraps produce very similar
confidence intervals, while the intervals produced by the empirical graphon
are noticeably wider.

\begin{figure}[ht]
  \centering
  \subfloat[]{\includegraphics[width=0.5\columnwidth]{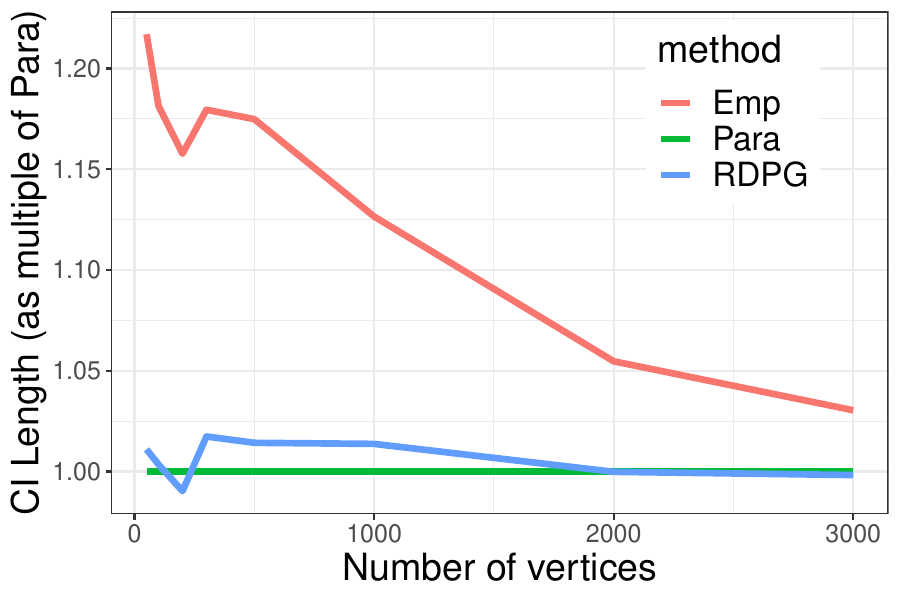}}
~~~
  \subfloat[]{\includegraphics[width=0.5\columnwidth]{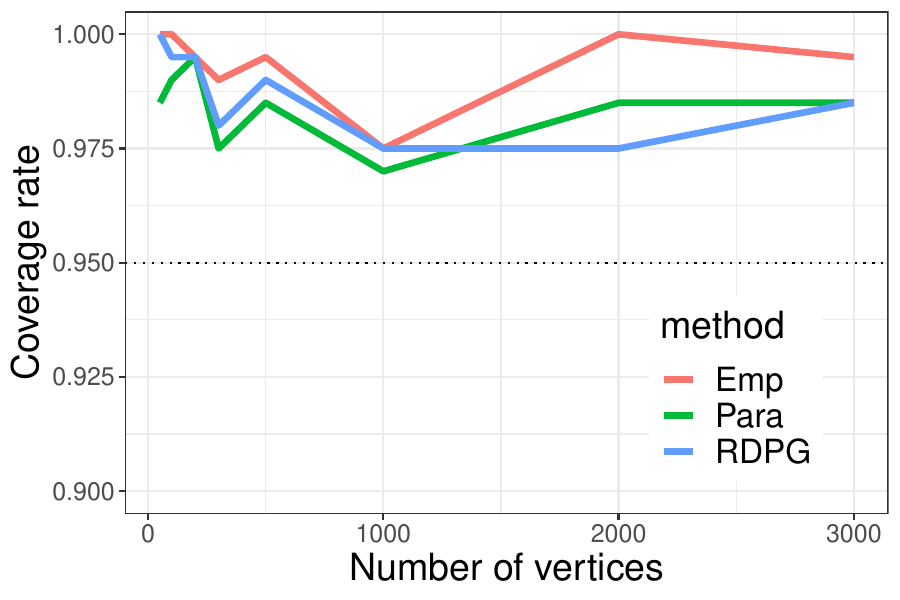}}
  \vspace{-3mm}
  \caption{(a) Average CI length as a multiple of the parametric bootstrap confidence interval length, over $200$ trials, for each value of $n$,
	produced by the empirical graphon (Emp, red),
	parametric bootstrap (Par, green) and	
	RDPG resampler (RDPG, blue).
	(b) Average coverage rates over the $200$ trials
	of the same three methods.}
  \label{fig:SPRDPG:summary}
\end{figure}

This is further borne out in Figure~\ref{fig:SPRDPG:summary}(a),
which shows the average confidence interval length over the $200$ trials,
for each of the three methods.
The plot shows the length of the confidence interval as a multiple of
the length of the interval produced by the parametric bootstrap.
We see that the RDPG bootstrap closely matches the length of the
parametric bootstrap, while the empirical graphon yields much wider intervals.
Figure~\ref{fig:SPRDPG:summary}(b) shows average coverage
rates of the three methods also aggregated over $200$ trials
for each choice of $n$.
We note that since the true expected average shortest path distance
is estimated via Monte Carlo, we computed coverage rates under
the milder requirement that a confidence interval was considered to have
covered the target if it overlapped the interval comprising
two standard errors of the mean of the Monte Carlo samples.
This adjustment did not appreciably change the coverage rates of the
three bootstrap methods.
We see from the plot that all three methods are overly conservative, but
the RDPG and parametric bootstraps somewhat less so, and they appear to improve more than the empirical graphon as $n$ increases.  
Given that the parametric bootstrap represents a sort of gold standard by fitting the true model, greatly narrowing the space of possible latent position distributions compared to the empirical and RDPG bootstraps,
it is quite encouraging that the RDPG bootstrap tracks the
performance of the parametric bootstrap so closely.

\section{Discussion and Conclusion} \label{sec:conclusion}
We have presented two methods for bootstrapping network data, applicable to any latent space model but studied in this paper under the random dot product graph. For network quantities expressible as U-statistics in the latent positions, our results in Section~\ref{sec:ustat} show that plugging in estimates of the true latent positions and proceeding with existing bootstrap techniques for U-statistics yields a distributionally consistent resampling procedure. Experimental evidence in Section~\ref{subsec:expt:uboot} supports this 
claim. 
By design, our resampling scheme is able to take advantage of existing
computational speedups for bootstrapping U-statistics, and thus provides a substantial computational improvement over existing approaches
to bootstrapping subgraph counts, which require expensive combinatorial enumeration.  

We have also proposed a method to resample whole networks by first estimating the latent positions and then drawing bootstrap samples from the empirical distribution of these estimates, followed by generating the network itself. We have shown that, again under the random dot product graph model, networks produced in this way are asymptotically
distributionally equivalent to the observed network from which
they are built. This distributional equivalence required defining the graph matching distance, which may be of
independent interest, as it provides a more intuitive notion of
graph distance than the more popular cut metric.

Directions for future work are many.   As alluded to in the paper, the core ideas presented here can be applied more broadly than the random dot product graph.   Our results in this paper can be extended trivially to the generalized random dot product graph \cite{RubPriTanCap2017} and graph root distributions \cite{Lei2018graphs}, but the basic ideas should work for any latent space model, as long as the latent positions can be accurately estimated. 
We leave an exploration of the precise analogues of our results for other latent space models to future work, along with investigating the extent to which the smoothness conditions required by the results of Section~\ref{sec:ustat} might be relaxed.

Our results in Section~\ref{sec:netboot} suggest several interesting
lines of inquiry. Firstly, experiments in Section~\ref{subsec:expt:netboot} confirm that our resampling
procedure improves over prior techniques for generating whole
network samples, but fails to obtain the desired coverage rate.  
Developing a correction for this is of great interest.
A small-scale experiment suggests that $\operatorname{BC}_a$
\citep{EfrTib1994}, a particularly popular bootstrap correction
technique, alleviates this issue, at least to an extent.
Unfortunately, the
large number of bootstrap samples required for this correction
is rather prohibitive in the network context,
and it would be preferable to develop a correction that explicitly takes
network structure into account.
More broadly, the ability to generate bootstrap replicates of networks
leads one to ask about the possibility of establishing network
analogues of classical bootstrap techniques such as the
$m$-out-of-$n$ bootstrap.

Finally, as discussed at the end of Section~\ref{sec:netboot},
convergence under the Wasserstein network distance
does not necessarily imply
convergence of other network statistics such as
$\sqrt{n}$-scaled subgraph densities.
It is possible that a stronger notion of distance will be required to ensure such convergences, one that eschews the fairly local perspective of the cut metric and graph matching distance in favor of more global measures of graph similarity.
A distance between networks that
considers path lengths or $k$-hop neighborhoods of individual vertices might better capture the global properties of networks that are
necessary to ensure that, for example, two networks have similar average path lengths.
\cite{ChaKubSch1998,BenIoa2018,TorSuaEli2018} present possible starting points for this line of work.
On the other hand, designing custom graph distances for every network statistic of interest is not ideal either, and we expect that future work in this direction will have to balance generality against improving rates for specific network statistics.

\bibliographystyle{plainnat}
\bibliography{biblio}

\newpage
\onecolumn
\appendix
\section*{Appendix} \label{sec:apx}

Here we provide supplemental proofs and technical details.
We note that in handling the competing goals of notational precision
and conformity with the existing literature,
we have opted for the latter, and as a result,
a few symbols are overloaded.
In particular, in the appendices that follow, $P$ will be used to denote either the subgraph density introduced in Section~\ref{sec:ustat} or the $n$-by-$n$ expectation of the adjacency matrix $A$ conditional on the latent positions.
Which of these two uses is intended will be clear from the context.
Similarly, the symbol $U$ is overloaded, denoting a U-statistic in some contexts and denoting an $n$-by-$d$ matrix with orthonormal columns in others.
Again, which of these two is intended will be clear from the context and from the fact that we subscript by $n$ (i.e., $U_n, \Uhat_n$, etc.) in the case of U-statistics, and leave plain $U, \Uhat$ to denote the matrices.

\section{Technical Results} \label{apx:gen}
We begin by collecting a handful of technical results from the existing
literature on random dot product graphs that will be useful in
the proofs below.

\begin{lemma}
\label{lem:rdpgfacts}
Let $(A,X) \sim \RDPG(F,n)$ for some $d$-dimensional inner product
distribution $F$.
Define $P = X X^T \in \bbR^{n \times n}$ so that
$\E[ A \mid X] = P$ and let
$\UP \SP \UP^T$ be the rank-$d$ eigendecomposition of $P$,
so that $\SP \in \bbR^{d \times d}$ is a diagonal matrix with entries
given by the eigenvalues
$\lambda_1(P) \ge \lambda_2(P) \ge \dots \ge \lambda_d(P) > 0$
and $\UP \in \bbR^{n \times d}$ has as its columns the $d$ corresponding
unit eigenvectors.
Similarly, let $\UA \SA \UA^T = \Xhat \Xhat^T \in \bbR^{n \times n}$
be the rank-$d$ approximation of
$A$ given by its top $d$ largest-magnitude eigenvalues and eigenvectors.
That is, let $\SA \in \bbR^{d \times d}$ be the diagonal matrix
with entries given by the $d$ largest-magnitude eigenvalues of $A$
and let $\UA \in \bbR^{n \times d}$
have as its columns the $d$ corresponding unit eigenvectors.
There exist constants $C_2 \ge C_1 > 0$ such that
with probability at least $1-Cn^{-2}$,
\begin{equation} \label{eq:specgrowth:1}
C_1 n \le \lambda_d(P) \le \dots \le \lambda_1(P) \le C_2 n
\end{equation}
\begin{equation} \label{eq:specgrowth:2}
\text{ and } \| A - P \| \le C \sqrt{n \log n}.
\end{equation}
Further, letting $Q \in \bbR^{d \times d}$
be the orthogonal matrix guaranteed by Lemma~\ref{lem:2toinfty},
for all suitably large $n$ it holds with probability at least
$1-Cn^{-2}$ that
\begin{equation} \label{eq:innerprod} 
\| Q - \UA^T \UP \|_F \le Cn^{-1} \log n, 
\end{equation}
\begin{equation} \label{eq:Qswap} 
\| Q \SA^{-1/2} - \SP^{-1/2} Q \|_F \le C n^{-3/2} \log^{1/2} n,
\end{equation}
\begin{equation} \label{eq:DKmodified} 
\| \UA Q - \UP \|_F \le \frac{ C \log^{1/2} n }{ \sqrt{n} }.
\end{equation}
\end{lemma}
\begin{proof}
Equations~\eqref{eq:specgrowth:1} and~\eqref{eq:specgrowth:2}
are Observations 1 and 2, respectively, in
\cite{LevAthTanLyzYouPri2017}.
Equation~\eqref{eq:innerprod} and~\eqref{eq:Qswap} follow from,
respectively, Proposition 16 and Lemma 17 in \cite{lyzinski15_HSBM},
with the slight alteration that we use the
spectral norm bound of \cite{Oliveira2009}
rather than that of \cite{LuPen2013}.
A proof of Equation~\eqref{eq:DKmodified} appears in the course of
the proof of Lemma 5 in \cite{LevAthTanLyzYouPri2017}.
We restate it here for the sake of completeness.

By Theorem 2 in \cite{YuWanSam2015},
there exists orthonormal $Q \in \bbR^{d \times d}$ such that
\begin{equation*}
\| \UA Q - \UP \|_F \le \frac{ C\sqrt{d} \| A - P \| }{ \lambda_d(P) }.
\end{equation*}
Applying Equation~\ref{eq:specgrowth:2} yields~\eqref{eq:DKmodified}.
\end{proof}

\begin{lemma} \label{lem:frobnorm1}
With notation as above, letting 
$Q \in \bbR^{d \times d}$ denote the orthogonal matrix guaranteed by
Lemma~\ref{lem:2toinfty}, with probability at least $1-Cn^{-2}$,
\begin{equation*}
\| A\UA(\SA^{-1/2} Q -  Q\SP^{-1/2}) \|_F \le C n^{-1/2} \log^{1/2} n.
\end{equation*}
\end{lemma}
\begin{proof}
Let $E = A - \UA \SA \UA^T$ be the residual after making the best rank-$d$
approximation to $A$. By definition, the eigenvectors of $E$ are
orthogonal to the columns of $\UA$, whence $E \UA = 0$, and thus
$A \UA = \UA \SA$.
\begin{equation*}
\| A\UA(\SA^{-1/2} Q -  Q\SP^{-1/2}) \|_F
= \| \UA \SA \UA^T \UA (\SA^{-1/2} Q -  Q\SP^{-1/2}) \|_F
\le \| \SA \| \| \SA^{-1/2} Q - Q\SP^{-1/2} \|_F.
\end{equation*}
Lemma~\ref{lem:rdpgfacts} bounds the spectral norm as $O(n)$,
the Frobenius norm as $O(n^{-3/2} \log^{1/2} n)$, which completes the proof.
\end{proof}

The following lemma is a generalization of
Lemma 10 of \cite{LyzSusTanAthPri2014} to the
case where $\Delta = \E_F X_1 X_1^T \in \bbR^{d \times d}$ may have
repeated eigenvalues.
\begin{lemma} \label{lem:frobnorm2}
With notation and setup as above,
Let $Q \in \bbR^{d \times d}$ be the orthogonal matrix guaranteed by
Lemma~\ref{lem:2toinfty}.
With probability at least $1-Cn^{-2}$,
\begin{equation*}
\| A(\UA Q - \UP)\SP^{-1/2} \|_F \le C n^{-1/2} \log n. log factor
\end{equation*}
\end{lemma}
\begin{proof}
Let $E = A - \UA \SA \UA^T$ as in the previous proof.
Taking $Q \in \bbR^{d \times d}$ to be as in Lemma~\ref{lem:2toinfty},
by the triangle inequality and basic properties of the Frobenius norm,
\begin{equation*} \begin{aligned}
\| A(\UA Q - \UP)\SP^{-1/2} \|_F
&= \| (\UA \SA \UA^T + E)(\UA Q - \UP) \SP^{-1/2} \|_F \\
&\le \| \UA \SA \| \| Q - \UA^T \UP \|_F \| \SP^{-1/2} \|
  + \| E \| \| \UA Q - \UP \|_F \| \SP^{-1/2} \|.
\end{aligned} \end{equation*}
Applying Lemma~\ref{lem:rdpgfacts},
with probability $1-Cn^{-2}$,
Equations~\eqref{eq:specgrowth:1} and~\eqref{eq:specgrowth:2},
both hold, so that
\begin{equation*}
\| \SP^{-1/2} \| = O(n^{-1/2}) \text{ and } \| \SA \| = O(n),
\end{equation*}
and Equation~\eqref{eq:innerprod} implies
\begin{equation*} \label{eq:Ugap:1}
\| \UA \| \| \SA \| \| Q - \UA^T \UP \|_F \| \SP^{-1/2} \|
\le  C n^{-1/2} \log n.
\end{equation*}
Similarly, since $\| E \| \le \|A-P\| = O(n^{1/2} \log n)$
by Equation~\eqref{eq:specgrowth:2},
Equation~\eqref{eq:DKmodified} implies that
\begin{equation*} \label{eq:Ugap:2}
\| E \| \| \UA Q - \UP \|_F \| \SP^{-1/2} \|
\le C n^{-1/2} \log^{1/2} n.
\end{equation*}
Thus, combining the above two displays,
\begin{equation*}
\| A(\UA Q - \UP)\SP^{-1/2} \|_F
\le C n^{-1/2} \log n,
\end{equation*}
as we set out to show.
\end{proof}

\section{Proof of Theorems~\ref{thm:ustat:consistent} and~\ref{thm:Uboot}}
\label{apx:inprob}
Here we provide detailed proofs of the results in Section~\ref{sec:ustat}.
Both rely on a second-order Taylor expansion
of the U-statistic evaluated at the true latent positions.
A similar argument appears in \cite{TanAthSusLyzPri2017}.
The main technical challenge here comes from the more complicated dependency
structure of U-statistics which in turn requires a more involved
indexing and counting argument.
The following two lemmas will prove useful in bounding the linear and
quadratic terms, respectively, in the Taylor expansion.
Throughout this appendix we use $\calC^n_m$ to denote the set of all
$m$-tuples of strictly increasing integers from $[n]$. That is,
$\calC^n_m = \{ (i_1,i_2,\dots,i_m) :
		i_1,i_2,\dots,i_m \in [n],
	1 \le i_1 < i_2 < \dots < i_m \le n \}$.
We write $\bbR^{\calC^n_m}$ to denote the set of vectors over
the reals with entries indexed by the $\binom{n}{m}$ elements of
$\calC^n_m$, so that if $v \in \bbR^{\calC^n_m}$, then
$v_c \in \bbR$ for each $c \in \calC^n_m$.
For $c \in \calC^n_m$ and $X \in \bbR^{n \times d}$,
we use $X_c$ to denote the $m$-by-$d$ matrix formed by stacking the rows of
$X$ whose indices appear in $c$.

\begin{lemma} \label{lem:wtdlinear}
Let $F$ be a $d$-dimensional inner product distribution
with $(A,X) \sim \RDPG(F,n)$ and let $\Xhat = \ASE(A,d)$.
Let $h : (\bbR^d)^m \rightarrow \bbR$ be a kernel function,
symmetric in its arguments,
satisfying Assumptions~\ref{assumu:hinvariant} and~\ref{assumu:smooth}.
Then there exists orthogonal matrix $Q \in \bbR^{d \times d}$ such that
for any fixed $v \in \bbR^{\calC^n_m}$,
with probability at least $1-Cn^{-2}$,
\begin{equation*} \label{eq:linearbound}
\left| \sum_{\ivec \in \calC^n_m} v_{\ivec}
	(\Xhat Q - X)_{\ivec}^T (\nabla h)(X_{\ivec}) \right|
\le C \max_{\ivec \in \calC^n_m} |v_{\ivec}| \binom{n-1}{m-1} \log n. 
\end{equation*}
\end{lemma}
\begin{proof}
Define the map
$\calT_{m} : \bbR^{n \times d} \rightarrow \bbR^{\calC^n_m \times md}$,
which transforms the matrix $Y \in \bbR^{n \times d}$
with rows $Y_i \in \bbR^d$ for $i=1,2,\dots,n$ into the matrix
$\Ytilde = \calT_m(Y) \in \bbR^{\calC^n_m \times md}$ as follows.
Indexing the $\binom{n}{m}$ rows of $\Ytilde = \calT_m(Y)$
by the $\binom{n}{m}$ elements of $\calC^n_m$,
define the row indexed by $\ivec = (i_1,i_2,\dots,i_m) \in \calC^n_m$ as
\begin{equation*}
\Ytilde_{\ivec} = \begin{bmatrix} Y_{i_1} \\
                                Y_{i_2} \\
                                \vdots \\
                                Y_{i_m} \end{bmatrix}^T 
	\in \bbR^{md}.
\end{equation*}
Denote by $\diag(v) \in \bbR^{\calC^n_m \times \calC^n_m}$
the diagonal matrix with entries given by the elements of $v$.
Define $\bM = \bM(\nabla h) \in \bbR^{\calC^n_m \times md}$
with the row indexed by $\ivec = (i_1,i_2,\dots,i_m) \in \calC^n_m$ given by
\begin{equation*}
\bM_{\ivec} = (\nabla h)(X_{\ivec})
= (\nabla h)\left( \begin{bmatrix} X_{i_1} \\
                        X_{i_2} \\
                        \vdots \\
                        X_{i_m} \end{bmatrix} \right) \in \bbR^{md}.
\end{equation*}
That is, the row of $\bM$ indexed by $\ivec \in \calC^n_m$
is the gradient of $h : \bbR^{md} \rightarrow \bbR$ evaluated at
$X_{\ivec} =[X_{i_1}^T, X_{i_2}^T,\dots, X_{i_m}^T ]^T$.
With these three definitions in hand, we have
\begin{equation*}
\sum_{\ivec \in \calC^n_m} v_{\ivec}
	(\Xhat Q - X)_{\ivec}^T (\nabla h)(X_{\ivec})
= \tr \bM^T \diag(v) \calT_m( \Xhat Q - X ).
\end{equation*}
Using the fact that
$X = \UP \SP^{1/2} = P \UP \SP^{-1/2}$
and
$\Xhat = \UA \SA^{1/2} = A \UA \SA^{-1/2}$,
adding and subtracting appropriate quantities, and using the
linearity of the trace and $\calT_m$, we have
\begin{equation*} \begin{aligned}
\sum_{\ivec \in \calC^n_m}
         & v_{\ivec} (\Xhat Q - X)_{\ivec}^T (\nabla h)(X_{\ivec}) \\
&= \tr\left[ \bM^T \diag(v)
	\calT_m\left(A\UA(\SA^{-1/2} Q -  Q\SP^{-1/2} ) \right) \right] \\
&~~~~~~+ \tr \left[
	\bM^T \diag(v) \calT_m\left( A(\UA Q - \UP) \SP^{-1/2} \right)
		\right]
       + \tr \left[
		\bM^T \diag(v) \calT_m\left( (A-P)\UP\SP^{-1/2} \right).
		\right]
\end{aligned} \end{equation*}
Applying the triangle inequality,
Cauchy-Schwarz and submultiplicativity, we have
\begin{equation} \label{eq:triangle:main}
\begin{aligned}
&\left| \sum_{\ivec \in \calC^n_m} v_{\ivec}
	(\Xhat Q - X)_{\ivec}^T (\nabla h)(X_{\ivec}) \right| \\
&~~~\le \| \bM \|_F \| \diag(v) \|
	 \left[ \| \calT_m( A\UA(\SA^{-1/2} Q -  Q\SP^{-1/2} )) \|_F
	+ \| \calT_m( A(\UA Q - \UP) \SP^{-1/2} ) \|_F \right] \\
&~~~~~~~~~+ \left| \tr\left[ \bM^T \diag(v) \calT_m( (A-P)\UP\SP^{-1/2} ) \right] \right|.
\end{aligned} \end{equation}

By definition of $\calT_m$, each row of $Z \in \bbR^{n \times d}$
appears in $\binom{n-1}{m-1}$ rows of $\calT_m(Z)$, and thus
$\| \calT_m( Z ) \|_F = \binom{n-1}{m-1}^{1/2} \| Z \|_F$.
Using this fact and applying Lemma~\ref{lem:frobnorm1},
\begin{equation} \label{eq:numbered4LL:1}
\begin{aligned}
\| \calT_m( A\UA(\SA^{-1/2} Q -  Q\SP^{-1/2} )) \|_F
&= \binom{n-1}{m-1}^{1/2} \| A\UA(\SA^{-1/2} Q -  Q\SP^{-1/2} ) \|_F \\
&\le C \binom{n-1}{m-1}^{1/2} n^{-1/2} \log n
\end{aligned} \end{equation}
Similarly, this time using Lemma~\ref{lem:frobnorm2},
\begin{equation} \label{eq:numbered4LL:2}
\begin{aligned}
\| \calT_m( A(\UA Q - \UP) \SP^{-1/2} ) \|_F
&\le \binom{n-1}{m-1}^{1/2} \| A (\UA Q - \UP) \SP^{-1/2} \|_F \\
&\le C \binom{n-1}{m-1}^{1/2} n^{-1/2} \log n.
\end{aligned} \end{equation}
Combining Equations~\eqref{eq:numbered4LL:1}
and~\eqref{eq:numbered4LL:2} and using the fact that
\begin{equation*}
\| \bM \|_F = \left( \sum_{\ivec \in \calC^n_m} \| (\nabla h)(X_{\ivec}) \|^2
                \right)^{1/2}
\le C \sqrt{ \binom{n}{m} }
\end{equation*}
by Assumption~\ref{assumu:smooth},
it follows that with probability at least $1-Cn^{-2}$,
\begin{equation} \label{eq:triangle:sub1}
\begin{aligned}
  \| \bM \|_F & \| \diag(v) \|
	\left[ \| \calT_m( A\UA(\SA^{-1/2} Q -  Q\SP^{-1/2} )) \|_F
            + \| \calT_m( A(\UA Q - \UP) \SP^{-1/2} ) \|_F \right] \\
&\le C n^{-1/2} \| \diag(v) \|
	 \left( \binom{n}{m} \binom{n-1}{m-1} \right)^{1/2} \log n
 \le C \| \diag(v) \| \binom{n-1}{m-1} \log n,
\end{aligned} \end{equation}
where we have used the fact that $m$ is assumed constant in $n$.

Returning to Equation~\eqref{eq:triangle:main}, it remains to bound
\begin{equation*}
 \left| \tr \left( \bM^T \diag(v) \calT_m( (A-P) \UP\SP^{-1/2} )
	\right) \right|.
\end{equation*}
By definition, for $\ivec \in \calC^n_m, k \in [m], s \in [d]$,
\begin{equation*}
\left( \calT_m( (A-P) \UP\SP^{-1/2} ) \right)_{\ivec,d(k-1) + s}
= \left[ (A-P) \UP\SP^{-1/2} \right]_{i_k, s}.
\end{equation*}
For any $\ivec = (i_1,i_2,\dots,i_m) \in \calC^n_m$ and $j \in [n]$,
if $j = i_k$ for some $k \in [m]$, define $\tau(j,\ivec) = k$.
With this notation in hand,
define the matrix $\bMtilde \in \bbR^{n \times d}$ by
\begin{equation*}
\bMtilde_{j,s} = \sum_{\ivec \in \calC^n_m : j \in \ivec}
        v_{\ivec}
	\bM_{\ivec, \tau(j,\ivec) + s},~~~j \in [n], s \in [d],
\end{equation*}
and note that
for some constant $C_{F,h} < \infty$ depending on $h$ and $F$
but not depending on $n$,
\begin{equation} \label{eq:Mtildebound}
\left| \bMtilde_{j,s} \right| \le C_{F,h} \binom{n-1}{m-1} \| \diag(v) \|,
\end{equation}
where we have again used Assumption~\ref{assumu:smooth}.
With this definition,
let $u_s \in \bbR^n$, $s=1,2,\dots,d$ be the eigenvectors of $P$
with non-zero eigenvalues (i.e., the columns of $\UP$),
so that $u_{s,i}$ denotes the $i$-th entry of the $s$-th eigenvector
of $P$. Then
\begin{align}
\tr \bM^T & \diag(v) \calT_m( (A-P) \UP\SP^{-1/2} )
= \sum_{\ivec \in \calC^n_m}
        v_{\ivec}
\sum_{j=1}^{md} \nonumber
	\bM_{\ivec,j} \calT_m\left( (A-P) \UP \SP^{-1/2} \right)_{\ivec,j}
        \nonumber \\
&= \sum_{\ivec \in \calC^n_m}
   v_{\ivec} \sum_{k=1}^m \sum_{s=1}^d
  \bM_{\ivec,m(k-1) + s} \left[ (A-P) \UP \SP^{-1/2} \right]_{i_k, s}
        \nonumber \\
&= \sum_{s=1}^d \lambda_s^{-1/2} \sum_{i=1}^n \sum_{j=1}^n
        \bMtilde_{i,s} (A-P)_{i,j} u_{s,j}
        \nonumber \\
&= \sum_{s=1}^d 2 \lambda_s^{-1/2} \sum_{1 \le i < j \le n}
        \bMtilde_{i,s} (A-P)_{i,j} u_{s,j} \label{eq:twosums}
- \sum_{s=1}^d \lambda_s^{-1/2} \sum_{i=1}^n
        \bMtilde_{i,s} P_{i,i} u_{s,i}.
\end{align}
The second term is bounded by
\begin{equation} \label{eq:tracebound1}
\begin{aligned}
\left| \sum_{s=1}^d \lambda_s^{-1/2} \sum_{i=1}^n
        \bMtilde_{i,s} P_{i,i} u_{s,i} \right|
&\le \sum_{s=1}^d \lambda_s^{-1/2}
        \left( \sum_{i=1}^n \bMtilde_{i,s}^2 P_{i,i}^2 \right)^{1/2}
        \| u_s \| \\
\le C d \| \diag(v) \| \binom{n-1}{m-1},
\end{aligned}
\end{equation}
where the first inequality follows from the Cauchy-Schwarz inequality,
and the second inequality follows from
Equation~\eqref{eq:Mtildebound}, the fact that $\| u_s \| = 1$,
the fact that $P_{i,i}^2 \le 1$, and Equation~\eqref{eq:specgrowth:1}.
For fixed $s \in [d]$, the sum over $1 \le i < j \le n$
in~\eqref{eq:twosums} is a sum of independent
mean-$0$ random variables. Hoeffding's inequality combined with
Equation~\eqref{eq:Mtildebound} yields
\begin{equation*} \begin{aligned}
\Pr&\left[ \left| \sum_{1 \le i < j \le n}
                \bMtilde_{i,s} (A-P)_{i,j} u_{s,j} \right| \ge t \right]
\le 2\exp\left\{ \frac{ -2 t^2 }{ \| u_s \|^2 \sum_{i=1}^n \bMtilde_{i,s}^2 }
                        \right\} \\
&~~~~~~~~~\le
2\exp\left\{ \frac{ -2 t^2 }{ C n \| \diag(v) \|^2 \binom{n-1}{m-1}^2 } \right\}.
\end{aligned} \end{equation*}
Taking $t = C \| \diag(v) \| \binom{n-1}{m-1} \sqrt{ n \log n}$
for suitably large constant $C > 0$, a union bound over all $s \in [d]$
implies that with probability $1- Cn^{-2}$,
it holds for all $s \in [d]$ that
\begin{equation*}
\left| \sum_{1 \le i < j \le n}
                \bMtilde_{i,s} (A-P)_{i,j} u_{s,j} \right| 
\le C \| \diag(v) \| \binom{n-1}{m-1} \sqrt{n \log n}.
\end{equation*}
Applying Equation~\eqref{eq:specgrowth:1} to bound $\lambda_s^{-1/2}$
and using Assumption~\ref{assumu:smooth},
it holds with probability at least $1-Cn^{-2}$ that
\begin{equation*}
\left| \sum_{s=1}^d 2 \lambda_s^{-1/2} \sum_{1 \le i < j \le n}
        \bMtilde_{i,s} (A-P)_{i,j} u_{s,j} \right|
\le C d \| \diag(v) \| \binom{n-1}{m-1} \log^{1/2} n.
\end{equation*}
Combining this with Equation~\eqref{eq:tracebound1},
both sums in \eqref{eq:twosums} are bounded by
$C d \| \diag(v) \| \binom{n-1}{m-1} \log^{1/2} n$, and it holds
with probability at least $1- Cn^{-2}$ that,
since $d$ is a constant,
\begin{equation*} 
\left| \tr \bM^T \calT_m( (A-P) \UP\SP^{-1/2} ) \right|
\le C \| \diag(v) \| \binom{n-1}{m-1} \log^{1/2} n.
\end{equation*}
Applying this and Equation~\eqref{eq:triangle:sub1} to
Equation~\eqref{eq:triangle:main}, we have
\begin{equation*}
\left| \sum_{\ivec \in \calC^n_m}
	v_{\ivec}
        (\Xhat Q - X)_{\ivec}^T (\nabla h)(X_{\ivec}) \right|
\le C \| \diag(v) \| \binom{n-1}{m-1} \log n ,
\end{equation*}
and the result follows by construction of $\diag(v)$.
\end{proof}

\begin{lemma} \label{lem:wtdquadratic}
Let $(A,X) \sim \RDPG(F,n)$ for some $d$-dimensional inner product distribution
$F$, so that $X_1,X_2,\dots,X_n \iid F$ and let $\Xhat = \ASE(A,d)$,
with rows given by $\Xhat_1,\Xhat_2,\dots,\Xhat_n \in \bbR^d$.
For each $\ivec \in \calC^n_m$,
let $Z_{\ivec} \in \bbR^{md}$ be some point on the
line segment connecting $(\Xhat Q)_{\ivec}$ and $X_{\ivec}$.
Suppose that $h : (\bbR^d)^m \rightarrow \bbR$ a kernel,
symmetric in its arguments, satisfying
Assumptions~\ref{assumu:hinvariant} and~\ref{assumu:smooth}.
Let $v \in \bbR^{\calC^n_m}$ be a fixed vector
and let $Q \in \bbR^{d \times d}$ be the orthogonal matrix guaranteed by
Lemma~\ref{lem:2toinfty}.
For all suitably large $n$, with probability at least $1-Cn^{-2}$,
\begin{equation*}
\left| \sum_{\ivec \in \calC^n_m} v_{\ivec}
        (\Xhat Q - X)_{\ivec}^T(\nabla^2 h)(Z_{\ivec})(\Xhat Q - X)_{\ivec}
	\right|
\le C \max_{\ivec \in \calC^n_m} |v_{\ivec}| \binom{n-1}{m-1} \log^2 n.
\end{equation*}
\end{lemma}
\begin{proof}
Let $Q \in \bbR^{d \times d}$ be the orthogonal matrix guaranteed to exist
with high probability by Lemma~\ref{lem:2toinfty}.
By Assumption~\ref{assumu:smooth},
Lemma~\ref{lem:2toinfty} implies that eventually
$X_{\ivec}, (\Xhat Q)_{\ivec} \in \calS$
for all $\ivec \in \calC^n_m$, and thus also
$Z_{\ivec} \in \calS$ for all $\ivec \in \calS$.
Applying the triangle inequality, Cauchy-Schwarz
and Assumption~\ref{assumu:smooth},
\begin{equation*} \begin{aligned}
&\left| \sum_{\ivec \in \calC^n_m}
	v_{\ivec}
        (\Xhat Q - X)_{\ivec}^T(\nabla^2 h)(Z_{\ivec})(\Xhat Q - X)_{\ivec}
        \right|
\le
  \sum_{\ivec \in \calC^n_m} \left|v_{\ivec} \right|
     \| (\Xhat Q - X)_{\ivec} \|_{\tti}^2 \| (\nabla^2 h)(Z_{\ivec}) \| \\
&~~~~~~~~~\le C_{F,h} \max_{\ivec \in \calC^n_m} |v_{\ivec}| \binom{n}{m} \| \Xhat Q - X \|_{\tti}^2 .
\end{aligned} \end{equation*}
Applying Lemma~\ref{lem:2toinfty} again, we have that
with probability at least $1-Cn^{-2}$,
\begin{equation} \label{eq:quadraticbound}
\left| \sum_{\ivec \in \calC^n_m} v_{\ivec}
        (\Xhat Q - X)_{\ivec}^T(\nabla^2 h)(Z_{\ivec})(\Xhat Q - X)_{\ivec}
        \right|
\le C \max_{\ivec \in \calC^n_m} |v_{\ivec}|
	\binom{n}{m} \frac{ \log^2 n }{ n }.
\end{equation}
Using the fact that $n^{-1}\binom{n}{m}$ = $m^{-1}\binom{n-1}{m-1}$
and $m$ is constant in $n$ completes the proof.
\end{proof}

\begin{proof}[Proof of Theorem~\ref{thm:ustat:consistent}]
We prove the convergence
$\sqrt{n}(\Uhat^*_n-\theta) \inlaw \calN(0,m^2 \zeta_1)$.
The proof for the ASE plug-in bootstrap
$\UBFhat^*_n$ follows by a similar argument,
and thus details are omitted.

For $\ivec \in \calC^n_m$ and $X \in \bbR^{n \times d}$, define
\begin{equation*}
X_{\ivec} = \begin{bmatrix} X_{i_1} \\
                        X_{i_2} \\
                        \vdots \\
                        X_{i_m} \end{bmatrix} \in \bbR^{md}.
\end{equation*}
Viewing the function $h: (\supp F)^m \rightarrow \bbR$ as
as $h : \bbR^{md} \rightarrow \bbR$
and applying a second-order multivariate Taylor expansion,
\begin{equation} \label{eq:taylor}
\begin{aligned}
\sqrt{n}&\left( \Uhat_n - U_n \right)
= \sqrt{n}\binom{n}{m}^{-1} \sum_{\ivec \in \calC^n_m}
        \left( h(\Xhat_{i_1},\dots,\Xhat_{i_m})
        - h(X_{i_1},\dots,X_{i_m}) \right) \\
&~~~= \sqrt{n} \binom{n}{m}^{-1} \sum_{\ivec \in \calC^n_m}
        (\Xhat Q - X)_{\ivec}^T (\nabla h)(X_{\ivec}) \\
&~~~~~~~~~~~~+ \frac{ \sqrt{n} }{ 2\binom{n}{m} } \sum_{\ivec \in \calC^n_m}
        (\Xhat Q - X)_{\ivec}^T(\nabla^2 h)(Z_{\ivec})(\Xhat Q - X)_{\ivec},
\end{aligned} \end{equation}
where $Q \in \bbO_d$ is the orthogonal matrix guaranteed by
Lemma~\ref{lem:2toinfty} and
$Z_{\ivec} \in \bbR^{md}$ lies on the line segment connecting
$(\Xhat Q)_{\ivec}$ and $X_{\ivec}$.
Lemma~\ref{lem:wtdlinear}, with $v_c = \sqrt{n}\binom{n}{m}^{-1}$
for all $c \in \calC^n_m$, implies
\begin{equation*}
\left| \sqrt{n} \binom{n}{m}^{-1}
\sum_{\ivec \in \calC^n_m}
        (\Xhat Q - X)_{\ivec}^T (\nabla h)(X_{\ivec}) \right|
\le \frac{ C \log n }{ \sqrt{n} }.
\end{equation*}
Lemma~\ref{lem:wtdquadratic} similarly implies that
\begin{equation*}
\left| \frac{ \sqrt{n} }{ 2 \binom{n}{m} }
\sum_{\ivec \in \calC^n_m}
        (\Xhat Q - X)_{\ivec}^T(\nabla^2 h)(Z_{\ivec})(\Xhat Q - X)_{\ivec}
\right|
\le \frac{ C \log^2 n }{ \sqrt{n} },
\end{equation*}
both holding with probability $1-Cn^{-2}$, and thus
\begin{equation*}
\left| \sqrt{n}\left( \Uhat_n - U_n \right) \right|
\le \frac{ C \log^2 n }{ \sqrt{n} }.
\end{equation*}
The Borel-Cantelli lemma implies
that $\sqrt{n}(\Uhat_n - U_n) \rightarrow 0$ almost surely,
as we wished to show.
\end{proof}

\begin{proof}[Proof of Theorem~\ref{thm:Uboot}]
By Slutsky's theorem, it will suffice for us to show that
\begin{equation} \label{eq:uboot:inprob}
\sqrt{n}(\Uhat^*_n-U^*_n) \inprob 0,
\end{equation}
since $\sqrt{n}(U^*_n - \theta) \inlaw \calN(0, m^2 \zeta_1)$
by our assumption that $U^*_n$ is distributionally consistent.
Applying an expansion similar to that in Equation~\eqref{eq:taylor} above,
we have
\begin{equation*} \begin{aligned}
\sqrt{n}(\Uhat^*_n-U^*_n)
&=
\sqrt{n} \sum_{\ivec \in \calC^n_m}
	\bbW_{\ivec}
	\left( h(\Xhat_{i_1},\Xhat_{i_2},\dots,\Xhat_{i_m})
	- h(X_{i_1},X_{i_2},\dots,X_{i_m}) \right) \\
&= \sqrt{n} \sum_{\ivec \in \calC^n_m}
	\bbW_{\ivec}
        (\Xhat Q - X)_{\ivec}^T (\nabla h)(X_{\ivec}) \\
&~~~~~~+ \frac{ \sqrt{n} }{2} \sum_{\ivec \in \calC^n_m}
	\bbW_{\ivec}
        (\Xhat Q - X)_{\ivec}^T(\nabla^2 h)(Z_{\ivec})(\Xhat Q - X)_{\ivec}
\end{aligned} \end{equation*}

Condition on the weight vector $\bbW \in \bbR^{\calC^n_m}$,
which is independent of $(A,X) \sim \RDPG(F,n)$.
Applying Lemmas~\ref{lem:wtdlinear} and ~\ref{lem:wtdquadratic}
with $v_{\ivec} = \sqrt{n}\bbW_{\ivec}/\binom{n}{m}$
implies that with high probability,
\begin{equation*}
\sqrt{n}(\Uhat^*_n-U^*_n)
\le
C\sqrt{n} \binom{n-1}{m-1}
\frac{ \max_{\ivec \in \calC^n_m} |\bbW_{\ivec}|}
		{ \binom{n}{m} } \log^2 n
\le \frac{ C\max_{\ivec \in \calC^n_m} |\bbW_{\ivec}| \log^2 n }{ \sqrt{n} },
\end{equation*}
where we have again used the fact that $m$ is constant in $n$.
Unconditioning, Assumption~\ref{assumw:growth} ensures that
\begin{equation*}
\sqrt{n}(\Uhat^*_n-U^*_n) = o(1),
\end{equation*}
which completes the proof.
\end{proof}

\section{Proof of Theorems~\ref{thm:sparse:ustat:consistent} and~\ref{thm:sparse:ustat:clt}}
\label{apx:sparse}

Here we give proofs of the sparsity results discussed in
Section~\ref{sec:ustat}.
We first need to ensure that in scaling the latent positions,
we do not break the recovery guarantees of the ASE.

\begin{lemma} \label{lem:sparse:2toinfty}
Let $\rho_n \rightarrow 0$ be a sparsity parameter,
satisfying $\rho_n n = \omega(\log n)$.
Let $F$ be a distribution on $\bbR^d$ 
with the property that for all suitably large $n$ it holds
for all $x,y \in \supp F$ that $\rho_n x^T y \in [0,1]$.
Let $X_1,X_2,\dots,X_n$ be drawn i.i.d.\ from $F$ and,
conditional on these $n$ points,
for all suitably large $n$ such that the Bernoulli success parameter makes
sense, generate symmetric adjacency matrix $A$
with independent entries $A_{ij} \sim \Bern(\rho_n X_i^T X_j)$.
Letting $\Xhat = \ASE(A,d)$, there exists a sequence of orthogonal matrices
$Q \in \bbO_d$ such that
\begin{equation*}
\| \Xhat Q - \sqrt{\rho}X \|_{2,\infty}
\le \frac{ \log n }{ \sqrt{ \rho n } }
\end{equation*}
\end{lemma}
\begin{proof}
Writing $\E[A \mid X] = \rho P = \rho X X^T \in \bbR^{n \times n}$ and
letting $\kappa(M)$ denote the ratio of the largest and smallest
non-zero singular values of matrix $M$ (i.e., the condition number
ignoring zero eigenvalues),
using Lemma 1 in \cite{LevLodLev2019}, there exists a matrix
$Q \in \bbO_d$ such that with high probability
\begin{equation} \label{eq:multinetbound}
\|\Xhat Q - \sqrt{ \rho } X \|_{2,\infty}
\le \frac{ C \|(A-\rho P) \UP \|_{2,\infty} }{ \sqrt{ \lambda_d(\rho P) } }
  + \frac{ C \| \UP^T (A-\rho P) \UP \|_F }{ \sqrt{ \lambda_d(\rho P) } }
  + \frac{ C \| A - \rho P \|^2 \kappa(\rho P) }{ \lambda_d^{3/2}( \rho P ) },
\end{equation}
provided that
\begin{equation} \label{eq:multinet:speccond}
  \| A - \rho P \| < C_0 \lambda_d(\rho P)
\end{equation}
for some nonnegative constant $C_0 < 1$.
Our assumption that $n \rho = \omega( \log n )$ is enough to ensure
that Theorem 3.1 in \cite{Oliveira2009} applies,
and it follows that
\begin{equation*} 
 \| A - \rho P \| = O( \sqrt{ \rho n \log n } ) 
\end{equation*}
By Equation~\eqref{eq:specgrowth:1} in Lemma~\ref{lem:rdpgfacts}, we have $\lambda_d(P) = \Theta( n )$,
whence $\lambda_d(\rho P) = \Theta( \rho n )$.
Since $\rho n = \omega( \log n)$ by assumption, it follows that
$\| A - \rho P \| = o( \lambda_d(\rho P) )$,
and we conclude that the bound in
Equation~\eqref{eq:multinet:speccond} holds eventually,
and thus so does the bound in Equation~\eqref{eq:multinetbound}.

We turn now to bounding the right-hand
side of Equation~\eqref{eq:multinetbound}.
For fixed $k, \ell \in [d]$,
consider the quantity
\begin{equation*}
R_{k\ell}
= \left( \UP^T (A-\rho P) \UP \right)_{k,\ell}
= \sum_{i,j} (A_{ij}-\rho P_{ij}) \UP_{i,k} \UP_{j,\ell}.
\end{equation*}
By Bernstein's inequality, for $t > 0$,
\begin{equation} \label{eq:bernstein:1}
\Pr\left[ \left| R_{k\ell} \right| > t \right]
\le 2\exp\left\{ \frac{ -t^2 }{ 2 \nu_{k,\ell} + 2t/3  } \right\},
\end{equation}
where 
\begin{equation*}
\nu_{k\ell}
= 2\sum_{i<j} \Var\left( (A_{ij}-\rho P_{ij}) \UP_{i,k} \UP_{j,\ell} \right)
= 2\sum_{i<j} \rho P_{ij} (1-\rho P_{ij})  \UP_{i,k}^2 \UP_{j,\ell}^2.
\end{equation*}
Using the fact that $0 \le P_{ij} \le 1$ for all $i,j$ and the fact that the columns of $\UP$ are orthonormal, we have $\nu_{k\ell} \le \rho \le 1$.
Thus, taking $t = C \log n$ in Equation~\eqref{eq:bernstein:1} for suitably large constant $C > 0$, we conclude that 
\begin{equation*}
\Pr\left[ \left| R_{k\ell} \right| > C\log n \right]
\le 2n^{-2}
\end{equation*}
Recalling that the dimension $d$ is constant,
a union bound over all $k,\ell \in [d]$ and an application of
the Borel-Cantelli Lemma implies that
\begin{equation} \label{eq:unitary1}
\| \UP^T (A-\rho P) \UP \|_F = O( \log n).
\end{equation}
A similar argument shows that
\begin{equation} \label{eq:unitary2}
\| (A-\rho P) \UP \|_{2,\infty} = O( \log n).
\end{equation}
Applying Equations~\eqref{eq:multinet:speccond},~\eqref{eq:unitary1}
and~\eqref{eq:unitary2}
to the right-hand side of Equation~\eqref{eq:multinetbound},
\begin{equation*}
\| \Xhat - \rho X \|_{2,\infty}
\le \frac{ C \log n }{ \sqrt{ \rho n } }
        + \frac{ C \kappa(\rho P) \rho n \log n }{ (\rho n)^{3/2} }.
\end{equation*}
Again using Equation~\eqref{eq:specgrowth:1} in
Lemma~\ref{lem:rdpgfacts},
$\kappa(\rho P) = \lambda_1(\rho P)/\lambda_d(\rho P) = O(1)$, and thus
\begin{equation*}
\| \Xhat - \rho X \|_{2,\infty} 
\le \frac{ C \log n }{ \sqrt{ \rho n } },
\end{equation*}
which completes the proof.
\end{proof}

\begin{proof}[Proof of Theorem~\ref{thm:sparse:ustat:consistent}]
Using arguments similar to the proof of Lemma~\ref{lem:sparse:2toinfty} above,
one can establish sparse analogues of
Lemmas~\ref{lem:rdpgfacts},~\ref{lem:wtdlinear} and~\ref{lem:wtdquadratic}.
Details are omitted.
Theorem~\ref{thm:sparse:ustat:consistent} then follows by precisely the
same line of argument as that used to prove
Theorem~\ref{thm:ustat:consistent}.
\end{proof}

\section{Proof of Theorem~\ref{thm:RDPGboot:wasser}} \label{apx:wasser}

\begin{proof}
Fix $\epsilon > 0$.
Since orthogonal transformation of the latent positions
does not change the graphs' distributions, we may assume without
loss of generality that $\dcirc_1(F_1,F_2) = d_1(F_1,F_2)$,
i.e., that $Q=I$ is the minimizer in Equation~\eqref{eq:def:dcirc}.
By definition of the Wasserstein distance $d_1$,
there exists a coupling $\nu$
of $X_1 \sim F_1$ and $Z_1 \sim F_2$ such that
\begin{equation} \label{eq:F0F1coupling}
 \int \| X_1 - Z_1 \| d \nu \le d_1(F_1,F_2) + \epsilon.
\end{equation}
We will use this coupling $\nu$ to construct a coupling of $A$ and $H$. Draw pairs
\begin{equation*}
(X_1,Z_1),(X_2,Z_2),\dots,(X_n,Z_n) \iid \nu.
\end{equation*}
It is a basic fact of Bernoulli random variables
that if $\xi_1 \sim \Bern(p_1)$ and $\xi_2 \sim \Bern(p_2)$, then
$d_1(\xi_1,\xi_2) \le |p_1-p_2|$.
Using this fact, conditional on $(X,Z)$, we can couple
$(A_{ij},H_{ij})$ for each $i < j$ so that
\begin{equation} \label{eq:bernoullicoupling}
\Pr[ A_{i,j} \neq H_{i,j} \mid X_i,X_j,Z_i,Z_j ] \le |X_i^TX_j - Z_i^TZ_j|.
\end{equation}
By construction,
$(A,X) \sim \RDPG(F_1,n)$ and $(H,Z) \sim \RDPG(F_2,n)$ marginally,
so this scheme yields a valid coupling of $A$ and $H$,
which we denote $(A,H) \sim \nutilde$, and thus
\begin{equation*} 
\WGM{p}^p(A,H) \le \int \dGM^p(A,H) d\nutilde(A,H).
\end{equation*}
By the definition of $\dGM$, Jensen's inequality,
and the fact that $A$ and $H$ are binary, we have
\begin{equation*}
\dGM^p(A,H) \le \left( \frac{1}{2}\binom{n}{2}^{-1} \| A - H \|_1 \right)^p
\le \binom{n}{2}^{-1} \sum_{i < j} |A-H|_{i,j}^p
= \binom{n}{2}^{-1} \sum_{i < j} |A-H|_{i,j},
\end{equation*}
whence
\begin{equation} \label{eq:dGM:int}
\int \dGM^p(A,H) d\nutilde(A,H)
\le \binom{n}{2}^{-1} \sum_{i < j} \int |A-H|_{i,j} d\nutilde
= \binom{n}{2}^{-1} \sum_{i < j}
	 \nutilde\left( \{ A_{ij} \neq H_{ij} \} \right) .
\end{equation}
Since Equation~\eqref{eq:bernoullicoupling} holds under the coupling $\nutilde$,
we have
\begin{equation*}
\nutilde\left( \{ A_{ij} \neq H_{ij} \} \right)
\le \int \int |X_i^T X_j - Z_i^T Z_j| d \nu(X_i,Z_i) d \nu (X_j,Z_j).
\end{equation*}
We can therefore further bound Equation~\eqref{eq:dGM:int} by
\begin{equation*} \begin{aligned}
\int \dGM^p(A,H) d\nutilde(A,H)
&\le \int \int |X_1^TX_2 - Z_1^T Z_2| d\nu(X_1,Z_1) d\nu(X_2,Z_2) \\
&\le \int \int \left(\| X_1 \|+\|Z_1\|\right)\|X_2 - Z_2\| d\nu d\nu \\
&\le 2\left( \int \|X_2 - Z_2\| d\nu \le d_1(F_1,F_2) + \epsilon \right),
\end{aligned} \end{equation*}
where we have used the fact that both $F_1$ and $F_2$, being inner product
distributions, have supports contained in the unit ball,
and the last inequality follows from Equation~\eqref{eq:F0F1coupling}.
Thus, we conclude that
\begin{equation*}
\WGM{p}^p(A,H) \le 2\left( d_1(F_1,F_2) + \epsilon \right),
\end{equation*}
and the result follows since $\epsilon > 0$ was arbitrary.
\end{proof}

\begin{proof}[Proof of Theorem~\ref{thm:RDPGboot:wasser}]
Let us first fix notation.
Recall that $(A,X) \sim \RDPG(F,n)$ and that
$(H,Z) \sim \RDPG(F,n)$ independently of $(A,X)$.
Let $F_n = n^{-1} \sum_{i=1}^n \delta_{X_i}$
denote the empirical distribution of the true latent positions of $A$,
and, conditional on $X$, let $(A^*,X^*) \sim \RDPG(F_n,n)$.
Letting $\Fhat_n$ denote the empirical distribution of the ASE estimates
$\Xhat_1,\Xhat_2,\dots,\Xhat_n$,
by definition of $\Ahat^*$,
we have that conditional on $A$,
$(\Ahat^*,\Xhat^*) \sim \RDPG(\Fhat,n)$ analogously.
By the triangle inequality,
\begin{equation} \label{eq:WGM:triangle}
\WGM{p}(H,\Ahat^*) \le \WGM{p}(A^*,H) + \WGM{p}(A^*,\Ahat^*).
\end{equation}
By Lemma~\ref{lem:GMLPwasser}, we have
\begin{equation} \label{eq:WGM:inter1}
\WGM{p}^p(A^*,H) \le 2d_1(F_n,F) = O(n^{-1/d} \log n),
\end{equation}
where we have used the fact
that $d$-dimensional product distributions have bounded support
(hence all moments of $X_1 \sim F$ are finite)
to apply Theorem 3.1 and Corollary 5.2 from \cite{Lei2018emp}
(with $q=\infty$ and $p=1$ in the notation of that paper)
to bound $d_1(F_n,F) = O(n^{-1/(2 \vee d)} \log n)$.
To bound $\WGM{p}(A^*,\Ahat^*)$, we will construct a coupling similar
to that in the proof of Lemma~\ref{lem:GMLPwasser}.

Letting $\xi = (\xi_1,\xi_2,\dots,\xi_n)$ be a vector of 
independent draws from the uniform distribution on $[n]$, we can write
$\Xhat^*_i = \Xhat_{\xi_i}$ for each $i=1,2,\dots,n$
and $X^*_i = X_{\xi_i}$ analogously.
Thus, we can couple the latent positions of $\Ahat^*$ and $A^*$ through the
random vector $\xi$. Then, conditional on $X$, $A$ and $\xi$,
we can further couple the entries of $\Ahat^*$
and $A^*$ via the same coupling construction used in
the proof of Lemma~\ref{lem:GMLPwasser} above, so that
\begin{equation*}
\Pr[ A^*_{i,j} \neq \Ahat^*_{i,j} \mid A,X,\xi] \le
| X_{\xi_i}^T X_{\xi_j} - \Xhat_{\xi_i}^T \Xhat_{\xi_j}|.
\end{equation*}
Letting $\nu$ denote the resulting joint measure on $(\xi,X,A,A^*,\Ahat^*)$,
\begin{equation} \label{eq:WGM:AstarAhat}
\WGM{p}^p(A^*,\Ahat^*)
\le \int \left( \frac{\|A^* - \Ahat^*\|_1}{n(n-1)} \right)^p d \nu
\le \int \frac{ \| A^* - \Ahat^* \|_1 }{ n(n-1) } d \nu,
\end{equation}
where we have
used Jensen's inequality and the fact that $A^*$ and $\Ahat^*$ are binary,
as in the proof of Lemma~\ref{lem:GMLPwasser}.
We will proceed to bound the integral on the right-hand side.
Let $E_n$ denote the event that the bound in Lemma~\ref{lem:2toinfty} holds.
On $E_n^c$, we can trivially bound $\dGM(A^*,\Ahat^*)$ by $1$.
Since $E_n$ depends only on $A$ and $X$, and the marginal distribution of $(A,X)$
under $\nu$ is $(A,X) \sim \RDPG(F,n)$ by construction of $\nu$,
Lemma~\ref{lem:2toinfty} implies $\nu(E_n^c) = O(n^{-2})$.
Thus,
\begin{equation*} \begin{aligned}
\int \frac{\|A^* - \Ahat^*\|_1}{n(n-1)} d \nu
&\le
\int_{E_n} \frac{\|A^* - \Ahat^*\|_1}{n(n-1)} d \nu + O(n^{-2}) \\
&= \sum_{i < j} \frac{ \nu\left( \left\{ A^*_{i,j} \neq \Ahat^*_{i,j} \right\}, E_n \right) }
		{ n(n-1) } + O(n^{-2}).
\end{aligned} \end{equation*}
By our construction of the coupling $\nu$, we have
\begin{equation*}
\nu\left(  \left\{ A^*_{i,j} \neq \Ahat^*_{i,j} \right\} \right)
\le |X_{\xi_i}^TX_{\xi_j} - \Xhat_{\xi_i}^T\Xhat_{\xi_j}|.
\end{equation*}
By Lemma~\ref{lem:2toinfty},
when $E_n$ holds, this difference of absolute values is bounded
by $O(n^{-1/2}\log n)$, and thus we have
\begin{equation*} \begin{aligned}
\int \frac{\|A^* - \Ahat^*\|_1}{n(n-1)} d \nu
&\le \binom{n}{2}^{-1}
    \sum_{i < j}
    \nu\left( \left\{ A^*_{i,j} \neq \Ahat^*_{i,j} \right\}, E_n \right)
	+ O(n^{-2}) \\
&= O(n^{-1/2} \log n) + O(n^{-2}) = O(n^{-1/2} \log n).
\end{aligned} \end{equation*}
Plugging this bound into Equation~\eqref{eq:WGM:AstarAhat},
we conclude that
\begin{equation*}
\WGM{p}^p(A^*,\Ahat^*) = O( n^{-1/2} \log n).
\end{equation*}
Applying this and Equation~\eqref{eq:WGM:inter1} to Equation~\eqref{eq:WGM:triangle}, we conclude that
\begin{equation*}
\WGM{p}^p(H,\Ahat^*) = O\left( (n^{-1/2} + n^{-1/d})\log n \right),
\end{equation*}
as we set out to show.
\end{proof}

\end{document}